\documentclass{article}

\def\TimesFont{} 
\def\ParSkip{} 
\usepackage{amssymb,amsmath,amsthm,bbm}
\usepackage{verbatim,float,url,dsfont}
\usepackage{graphicx,subcaption,psfrag}
\usepackage{algorithm,algorithmic}
\usepackage{mathtools,enumitem}
\usepackage{multirow}
\usepackage{ragged2e}
\usepackage{xr-hyper}
\usepackage{array}

\usepackage[utf8]{inputenc} 
\usepackage[T1]{fontenc}    
\usepackage{booktabs}       
\usepackage{nicefrac}         
\usepackage{microtype}      

\usepackage[margin=1in]{geometry}
\usepackage[round]{natbib}
\usepackage[colorlinks=true,citecolor=blue,urlcolor=blue,linkcolor=blue]{hyperref}

\ifdefined\TimesFont
\usepackage{times} 
\fi

\ifdefined\ParSkip
\usepackage{parskip} 
\fi

\newtheorem{theorem}{Theorem}
\newtheorem{lemma}{Lemma}

\newtheorem{proposition}{Proposition}
\theoremstyle{definition}
\newtheorem{remark}{Remark}

\newtheorem{example}{Example}
\newtheorem*{remark*}{Remark}  

\newtheorem*{assumption*}{\assumptionnumber}
\providecommand{\assumptionnumber}{}


\makeatletter
\newcommand*\rel@kern[1]{\kern#1\dimexpr\macc@kerna}
\newcommand*\widebar[1]{%
  \begingroup
  \def\mathaccent##1##2{%
    \rel@kern{0.8}%
    \overline{\rel@kern{-0.8}\macc@nucleus\rel@kern{0.2}}%
    \rel@kern{-0.2}%
  }%
  \macc@depth\@ne
  \let\math@bgroup\@empty \let\math@egroup\macc@set@skewchar
  \mathsurround\z@ \frozen@everymath{\mathgroup\macc@group\relax}%
  \macc@set@skewchar\relax
  \let\mathaccentV\macc@nested@a
  \macc@nested@a\relax111{#1}%
  \endgroup
}
\makeatother

\DeclareMathOperator{\Var}{Var}

\DeclareMathOperator{\sign}{sign}

\def\E{\mathbb{E}}

\def\R{\mathbb{R}}

\def\cF{\mathcal{F}}

\def\cY{\mathcal{Y}}

\DeclareMathOperator{\KL}{KL}

\usepackage{tikz}
\newcommand{\tikzmark}[1]{\tikz[overlay,remember picture,baseline]
  \node[anchor=base, inner sep=0pt] (#1) {\strut};}
\usepackage{pgfplots}
  \pgfplotsset{compat=1.17}
\usepackage{listofitems}
\usepackage{xintexpr}
\usepackage{readarray}

\title{Asymmetric Penalties Underlie Proper Loss Functions in Probabilistic
  Forecasting}
\author{Erez Buchweitz \and João Vitor Romano \and Ryan J. Tibshirani}
\date{}

\begin{document}
\maketitle

\begin{abstract}
Accurately forecasting the probability distribution of phenomena of interest is
a classic and ever more widespread goal in statistics and decision theory.  In
comparison to point forecasts, probabilistic forecasts aim to provide a more
complete and informative characterization of the target variable. This endeavor
is only fruitful, however, if a forecast is ``close'' to the distribution it
attempts to predict. The role of a loss function---also known as a scoring
rule---is to make this precise by providing a quantitative
measure of proximity between a forecast distribution and target random
variable. Numerous loss functions have been proposed in the literature, with a
strong focus on proper losses, that is, losses whose expectations are minimized
when the forecast distribution is the same as the target. In this paper, we show
that a broad class of proper loss functions penalize \emph{asymmetrically}, in
the sense that underestimating a given parameter of the target distribution can
incur larger loss than overestimating it, or vice versa. Our theory covers many
popular losses, such as the logarithmic, continuous ranked probability,
quadratic, and spherical losses, as well as the energy and threshold-weighted
generalizations of continuous ranked probability loss. To complement our theory,
we present experiments with real epidemiological, meteorological, and retail
forecast data sets. Further, as an implication of the loss asymmetries revealed
by our work, we show that hedging is possible under a setting of distribution
shift.
\end{abstract}

\section{Introduction}
\label{sec:introduction}

In probabilistic forecasting, the goal is to predict the distribution of a
target variable, rather than a particular parameter of that distribution, such
as its mean or a single quantile, which is termed point forecasting. The
pursuit of probabilistic forecasts has been promoted based on philosophical
grounds, as an adequate form of expressing the inherent uncertainty in
predicting the target variable \citep{dawid1984present,
gneiting2007strictly}. From a practical viewpoint, a probabilistic forecast
provides more detailed information on the possible realizations of the target,
which is generally considered useful for decision making
\citep{jordan2011operational, jolliffe2012forecast,
  cramer2022eval}. Accordingly, in various scientific disciplines, probabilistic
forecasts are widespread and growing in use \citep{gneiting2014probabilistic},
with some examples being the prediction of
floods \citep{pappenberger2011impact},
 earthquakes \citep{schorlemmer2018collaboratory},
epidemics \citep{cramer2022eval},
energy usage \citep{hong2016probabilistic},
population growth \citep{raftery2023probabilistic},
and inflation \citep{galbraith2012assessing}.

Loss functions measure the discrepancy between a prediction and an observed
target. In probabilistic forecasting, each prediction is an entire probability
distribution, whereas the target is a random variable. Hence, a loss function
assigns a numeric value to the discrepancy between a real value (in case the
target variable is real) and a probability distribution from which it has
purportedly originated. The study of loss functions---commonly called scoring
rules in probabilistic forecasting---originated in weather
forecasting \citep{brier1950verification, winkler1968good,
  matheson1976scoring}, intertwined with the development of subjective
probability \citep{good1952rational, definetti1970theory,
  savage1971elicitation}, and has largely revolved around propriety. A loss
function is deemed proper if the least loss, in expectation over all possible
realizations of the target, is incurred when the forecast equals the
distribution of the target \citep{gneiting2007strictly}.

It is commonly held that proper losses encourage ``honest forecasting''
\citep{gneiting2007strictly, parry2012proper}; formally, this is guaranteed only
in the special (rare) case that the forecaster knows \emph{with certainty} what
the distribution of the target is, as then they can do no better than to
forecast it. In any case, this forms the basis of the wide appeal and
subsequently the wide prevalence of proper losses, in both theory and practice.

A practitioner looking to fit or evaluate probabilistic forecasts must choose
which loss function to optimize. For instance, it is common to fit a forecasting
model by minimizing the average loss over a training set \citep{rasp2018neural},
or to choose from among a collection of forecasts by minimizing the average loss
on a test set \citep{cramer2022eval}. \citet{dawid2007geometry} lays out a
method for constructing proper losses in which propriety is tied to Bayes
(optimal) actions in a given decision-making task. However, it remains that
practitioners largely prefer to choose one out of a plethora of losses that have
been proposed in the literature for general purpose, or a combination thereof
(as exhibited in the domain-focused references on forecasting given
previously). In the absence of a definitive theory as to which loss is
preferable in a given situation, some authors have recommended choosing one
whose well-known properties appear to be aligned with the use case
\citep{winkler1996scoring, gneiting2007strictly}. The study of the properties of
each particular loss function has therefore become a major theme in
probabilistic forecasting, and we contribute to this literature by
characterizing \emph{asymmetries}: which forecasting errors are awarded lesser
or greater loss, by particular loss functions.

By definition, any proper loss function will favor the true distribution of the
target by awarding it the minimal loss, on average over all possible
realizations of the target. But in practice, forecasts are rarely equal to the
true distribution of the target due to model misspecification, distribution
shift and for other reasons. Which forecast will then be awarded the least
expected loss depends on how a particular loss function penalizes different
kinds of errors. Figure \ref{fig:heatmaps-normal} displays the expected loss
when the forecast and target distribution are both normal, for two of the most
commonly used loss functions in practice. Given a choice between two
forecasts where the forecast variance is either double or half the target
variance, \emph{logarithmic loss} will favor doubling the variance, choosing a
flat, less informative forecast, while \emph{continuous ranked probability (CRP)
  loss} will favor halving the variance, providing for a sharp overconfident
forecast. (In both cases the mean is correctly specified. For definitions of
those and other loss functions, see Section \ref{sec:preliminaries}.)

The same findings can be recreated on real data from the Covid-19 Forecast Hub
\citep{cramer2022us}, as shown in Figure \ref{fig:heatmaps-covid}. (For
details on the experimental setup, see Section \ref{sec:empirical-results}.) We
highlight another effect appearing in the left-most plot in Figure
\ref{fig:heatmaps-covid}: given a choice between two forecasts, either shifted
upward by one unit, or shifted downward by the same amount, logarithmic loss
will favor the downshifted forecast. We attribute this asymmetry effect
particularly to right-skewed forecasts, which Covid-19 mortality distributions
typically are. Finally, notice in Figure \ref{fig:heatmaps-covid} that the
target mean and variance do not minimize logarithmic loss. This is because the
forecast location-scale family is misspecified, meaning that it does not equal
the target location-scale family. A setting of misspecification in which
minimizers of loss deviate from the target is studied analytically in Section
\ref{sec:hedging}, and further research is warranted.

\begin{figure}
\vspace{-46pt}
\centering
\begin{tabular}{cc}
Logarithmic loss & CRP loss \\
\includegraphics[width=0.28\textwidth]{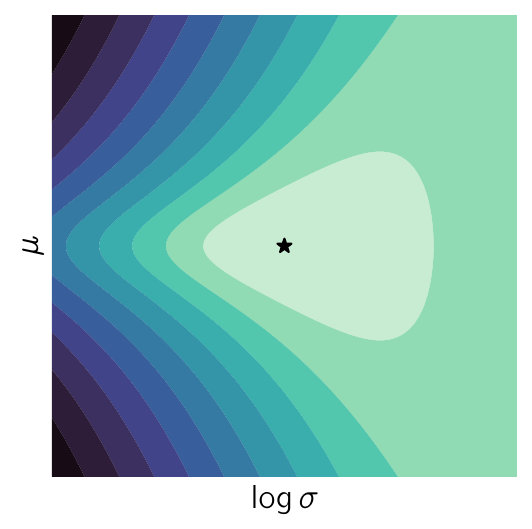} &
\includegraphics[width=0.28\textwidth]{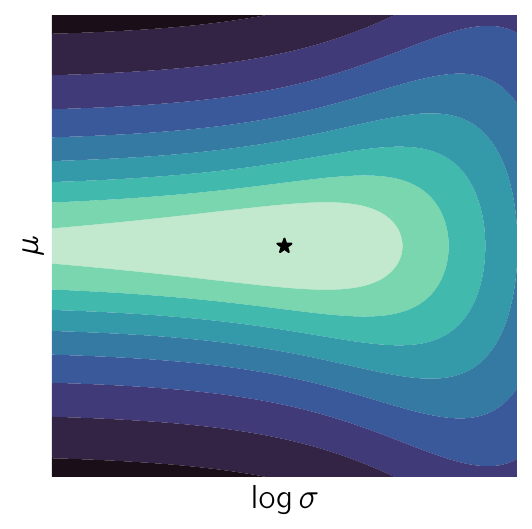}
\end{tabular}
\vspace{-15pt}
\caption{\small Expected logarithmic and CRP losses for a fixed standard normal
  target and normal forecasts with varying location $\mu$ and scale
  $\sigma$. A lighter color represents a lower loss, with minimum achieved at
  the star, where the forecast distribution is also a standard normal. When the
  location is correctly specified, logarithmic loss penalizes underestimating
  the scale more than overestimating it, whereas the opposite is true for
  CRP. When the scale is correctly specified, both losses penalize symmetrically
  on the location.}
\label{fig:heatmaps-normal}

\medskip
\begin{tabular}{cc}
Logarithmic loss & CRP loss \\
\includegraphics[width=0.28\textwidth]{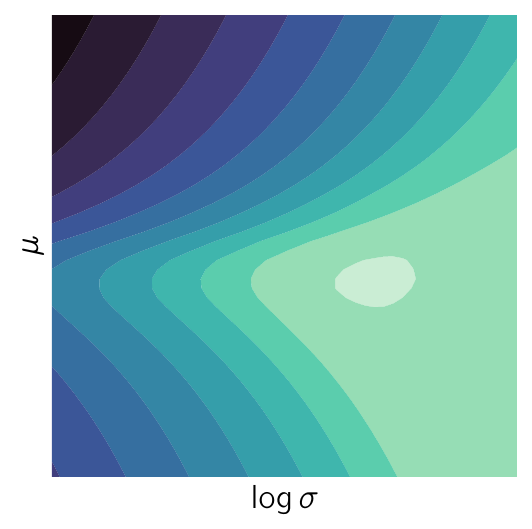} &
\includegraphics[width=0.28\textwidth]{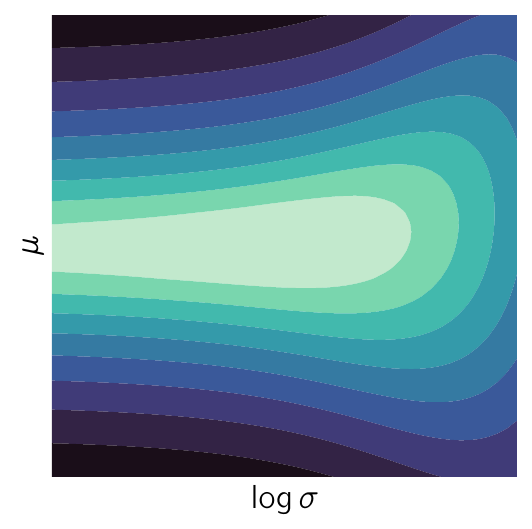}
\end{tabular}
\vspace{-15pt}
\caption{\small Average logarithmic and CRP losses over targets and forecasts
  from the Covid-19 Forecast Hub (for details see Section
  \ref{sec:empirical-results}). The qualitative assessment for CRP loss is
  exactly the same as in the normal case; for logarithmic loss there is an
  additional location asymmetry that is due to the forecasts being
  right-skewed.}
\label{fig:heatmaps-covid}

\bigskip
{\scriptsize
\begin{tabular}{l@{\hspace{12em}}l}
  \normalsize{Logarithmic loss}        & \normalsize{CRP loss} \\
  \midrule
  1. COVIDhub-ensemble                \tikzmark{a1} & \tikzmark{a2} 1. COVIDhub-ensemble                \tikzmark{a3} \\
  2. UMass-MechBayes                  \tikzmark{d1} & \tikzmark{b2} 6. SteveMcConnell-CovidComplete     \tikzmark{b3} \\
  3. Karlen-pypm                      \tikzmark{c1} & \tikzmark{c2} 3. Karlen-pypm                      \tikzmark{c3} \\
  4. OliverWyman-Navigator            \tikzmark{h1} & \tikzmark{d2} 2. UMass-MechBayes                  \tikzmark{d3} \\
  5. epiforecasts-ensemble1           \tikzmark{g1} & \tikzmark{e2} \textcolor{gray}{8. GT-DeepCOVID}   \tikzmark{e3} \\
  6. SteveMcConnell-CovidComplete     \tikzmark{b1} & \tikzmark{f2} 9. MOBS-GLEAM\_COVID                \tikzmark{f3} \\
  \textcolor{gray}{7. CMU-TimeSeries} \tikzmark{i1} & \tikzmark{g2} 5. epiforecasts-ensemble1           \tikzmark{g3} \\
  \textcolor{gray}{8. GT-DeepCOVID}   \tikzmark{e1} & \tikzmark{h2} 4. OliverWyman-Navigator            \tikzmark{h3} \\
  9. MOBS-GLEAM\_COVID                \tikzmark{f1} & \tikzmark{i2} \textcolor{gray}{7. CMU-TimeSeries} \tikzmark{i3} \\
  10. USC-SI\_kJalpha                  \tikzmark{j1} & \tikzmark{j2} 10. USC-SI\_kJalpha                  \tikzmark{j3} \\
\end{tabular}
\begin{tikzpicture}[overlay, remember picture]
  \draw[-,thick,orange] (a1) to [out=0,in=180] (a2);
  \draw[-,thick,orange] (b1) to [out=0,in=180] (b2);
  \draw[-,thick,orange] (c1) to [out=0,in=180] (c2);
  \draw[-,thick,orange] (f1) to [out=0,in=180] (f2);
  \draw[-,thick,cyan]   (d1) to [out=0,in=180] (d2);
  \draw[-,thick,cyan]   (g1) to [out=0,in=180] (g2);
  \draw[-,thick,cyan]   (h1) to [out=0,in=180] (h2);
  \draw[-,thick,cyan]   (j1) to [out=0,in=180] (j2);
\end{tikzpicture}}
\vspace{-5pt}
\caption{\small Ranking of top forecasters in the Covid-19 Forecast Hub via
  logarithmic and CRP losses. Each line connects the same forecaster over the
  two ranked lists, and orange (cyan) lines identify the forecasters with lowest
  (highest) standardized variance.}
\label{tab:covid-ranking}

\bigskip\bigskip
\begin{tikzpicture}[yscale=-1]
  \def\height{4.9}
  \def\margin{0.2*\textwidth}
  \useasboundingbox (0,0) rectangle (\textwidth,\height);
  \readdef{res/logranklowvar.txt}{\logranklowvar}
  \readdef{res/crpranklowvar.txt}{\crpranklowvar}
  \readdef{res/logrankhighvar.txt}{\logrankhighvar}
  \readdef{res/crprankhighvar.txt}{\crprankhighvar}
  \readdef{res/logforecasters.txt}{\logforecasters}
  \readdef{res/crpforecasters.txt}{\crpforecasters}
  \setsepchar{,}
  \readlist\logforecasters{\logforecasters}
  \readlist\crpforecasters{\crpforecasters}
  \pgfmathsetmacro{\collog}{\margin}
  \pgfmathsetmacro{\colcrp}{\textwidth - \margin}
  \draw (\collog pt,0) -- (\collog pt,\height) node[pos=-0.01, above] {Logarithmic loss};
  \draw (\colcrp pt,0) -- (\colcrp pt,\height) node[pos=-0.01, above] {CRP loss};
  \foreach \n in {0,...,14} {
    \pgfmathsetmacro{\ilow}{{\logranklowvar}[\n]}
    \pgfmathsetmacro{\jlow}{{\crpranklowvar}[\n]}
    \pgfmathsetmacro{\ihigh}{{\logrankhighvar}[\n]}
    \pgfmathsetmacro{\jhigh}{{\crprankhighvar}[\n]}
    \draw[-,thick,orange] (\collog pt,{\ilow/10}) to [out=0,in=180] (\colcrp pt,{\jlow/10});
    \draw[-,thick,cyan] (\collog pt,{\ihigh/10}) to [out=0,in=180] (\colcrp pt,{\jhigh/10});
    \node[anchor=west,align=left] at (1.4,{\ilow/10}) {\fontsize{3.3}{4}\selectfont \logforecasters[\ilow+1]};
    \node[anchor=west,align=left] at (1.4,{\ihigh/10}) {\fontsize{3.3}{4}\selectfont \logforecasters[\ihigh+1]};
    \node[anchor=west,align=left] at (\colcrp pt,{\jlow/10}) {\fontsize{3.3}{4}\selectfont \crpforecasters[\jlow+1]};
    \node[anchor=west,align=left] at (\colcrp pt,{\jhigh/10}) {\fontsize{3.3}{4}\selectfont \crpforecasters[\jhigh+1]};
  }
\end{tikzpicture}
\vspace{-5pt}
\caption{\small Ranking of top forecasters in the M5 Uncertainty Competition via
  logarithmic and CRP losses. Lines again connect the same forecasters over the
  ranked lists, and orange (cyan) lines identify the forecasters with lowest
  (highest) standardized variance.}
\label{fig:m5-ranking}
\end{figure}

The effects described above may lead, among other things, to:
\begin{enumerate}
\item Chosen forecasts having a systematic tendency to be overly flat, sharp, or
  shifted, depending on the loss function used to choose them;
\item Forecasters being incentivized to report overly flat, sharp, or shifted
  forecasts, again depending on the loss used to evaluate them.
\end{enumerate}

The first consequence above is demonstrated in Figures \ref{tab:covid-ranking}
and \ref{fig:m5-ranking}, which depict results on real data from the Covid-19
Forecast Hub and the M5 Uncertainty Competition
\citep{makridakis2022uncertainty}, respectively. In both cases, forecasters with
high standardized variance (meaning that they produced relatively flat
forecasts) were generally ranked at least as high under logarithmic loss
(meaning that they received relatively low loss) as under CRP loss. Conversely,
CRP loss generally ranked forecasters with low standardized variance at least as
high as logarithmic loss did. Practitioners ought to be aware of the asymmetries
described in this paper pertaining to these and other commonly used loss
functions. The choice of loss function can affect the ranking of forecasts in
predictable ways. The second of the consequences above, detailing incentives
imposed upon the forecasters by different loss functions, is discussed in
Section \ref{sec:hedging}.

\subsection{Summary of main results}

We now summarize the main results of this paper, on asymmetries inherent to
logarithmic, CRP, quadratic, spherical, energy, and Dawid--Sebastiani (DS) losses
(Section \ref{sec:preliminaries} gives their definitions). We denote by
$\ell(F,G)$ the expected loss when $F$ is the forecast and $G$ is the target
distribution. Our first main result (Section \ref{sec:main-scale}), pertains to
scale families.

\medskip
\begin{theorem}[Scale family]
\label{thm:scale-family}
Let $\{G_\sigma:\sigma>0\}$ be a scale family. Fix $G=G_1$ and $\sigma>1$. The
following holds.
\begin{itemize}
\item For CRP and energy losses, $\ell(G_\sigma,G) > \ell(G_{1/\sigma},G)$.
\item For quadratic and DS losses, $\ell(G_\sigma,G) < \ell(G_{1/\sigma},G)$.
\item For spherical loss, $\ell(G_\sigma,G) = \ell(G_{1/\sigma},G)$.
\end{itemize}
\end{theorem}
\smallskip

Supposing the target variable $Y$ follows a distribution $G$, we define
$G_\sigma$ as the distribution of $\sigma Y$. For CRP and energy losses, the
expected loss is larger when forecasting $G_\sigma$ versus
\smash{$G_{1/\sigma}$} for any value of $\sigma > 1$. Recalling $G=G_1$, these
loss functions thus favor underestimating the scale parameter compared to
overestimating it by the same amount, on a logarithmic scale. In other words,
facing a choice between a flat, less informative forecast $G_\sigma$ and a sharp
overconfident forecast \smash{$G_{1/\sigma}$}, CRP and energy losses will always
favor the latter, regardless of the particular scale family. Conversely,
quadratic and DS losses will always favor the former: the flat, less informative
forecast in a scale family. Spherical loss will favor neither, however this
should not be taken at face value, as will become clearer later. For
threshold-weighted CRP loss, the asymmetry depends on the weight function (see
Section \ref{sec:tw-crp}). We find Theorem \ref{thm:scale-family} in agreement
with the empirical findings in Figure \ref{fig:heatmaps-normal}. Importantly,
Figure \ref{fig:heatmaps-covid} presents empirical evidence that a similar
phenomenon can occur when the forecast and target distribution are not of the
same family.

Logarithmic loss is notably absent from Theorem \ref{thm:scale-family}. For this
loss, the direction of asymmetry may differ between scale families. Our second
main result (Section \ref{sec:main-exponential}) characterizes the asymmetry
inherent to logarithmic loss, with respect to exponential families.

\medskip
\begin{theorem}[Exponential family]
\label{thm:exponential-family}
Let $\{G_\theta: \theta >0\}$ be a minimal exponential family. Fix $G=G_1$,
$\theta>1$, and denote by $\ell$ the logarithmic loss. The following holds.
\begin{itemize}
\item If the family is a normal, exponential, Laplace, Weibull, or (generalized)
  gamma scale family, or a log-normal log-scale family, then
  $\ell(G_{\theta},G) < \ell(G_{1/\theta},G)$.
\item If the family is an inverse gamma scale family, then
  $\ell(G_{\theta},G) > \ell(G_{1/\theta},G)$.
\item If the family is a (generalized) gamma, Pareto, inverse Gaussian, or beta
  shape family, or a Poisson rate family, then
  $\ell(G_{\theta},G) > \ell(G_{1/\theta},G)$.
\end{itemize}
\end{theorem}
\smallskip

In most (but not all) of the examples of scale families that we were able to
find, logarithmic loss favors overestimating the scale rather than
underestimating it by the same amount, on a logarithmic scale as before. This is
in agreement with Figure \ref{fig:heatmaps-normal} (which examines the normal
case) and \ref{fig:heatmaps-covid} (the misspecified case where $F,G$ are not
of the same family). The converse is true for exponential families with
so-called shape parameters, typically right-skewed, in which logarithmic loss
favors underestimating the shape. We believe the empirical results pertaining to
asymmetry in right-skewed location families, which are favored underestimated in
Figure \ref{fig:heatmaps-covid}, may be understood in this light.

We complement the above results on scale and exponential families with a result
on location families.

\medskip
\begin{theorem}[Location family]
\label{thm:location-family}
Let $\{G_\mu:\mu\in\R\}$ be a location family. Fix $G=G_0$. The following holds.
\begin{itemize}
\item For CRP, quadratic, spherical, energy, and DS losses,
  $\ell(G_\mu,G) = \ell(G_{-\mu},G)$.
\item For logarithmic loss, $\ell(G_\mu, G) = \ell(G_{-\mu},G)$, provided $G$ is
  symmetric.
\end{itemize}
\end{theorem}
\smallskip

Supposing the target variable $Y$ follows a distribution $G$, we define
$G_\mu$ as the distribution of $Y+\mu$. CRP, quadratic, spherical, energy, and
DS losses do not favor underestimating nor overestimating the location, in any
location family. Logarithmic loss behaves the same way, provided $G$ is
symmetric. For threshold-weighted CRP loss, the asymmetry depends on the weight
function (see Section \ref{sec:tw-crp}). These results are in line with Figure
\ref{fig:heatmaps-normal}, and Figure \ref{fig:heatmaps-covid} again provides
empirical evidence that the same can be true when the forecast and target
distribution are not of the same family (recall, the asymmetry for logarithmic
loss with respect to location can be interpreted via Theorem
\ref{thm:exponential-family}).

The results in Theorems \ref{thm:scale-family}--\ref{thm:location-family} still
hold when we replace each $\ell(F,G)$ by the divergence
$d(F,G) = \ell(F,G) - \ell(G,G)$. Common proper losses induce well-known
divergences, such as the Kullback--Leibler divergence (which is induced by
logarithmic loss), Cramér distance (CRP loss), energy distance (energy loss),
$L^2$ distance between densities (quadratic loss), and so on. These divergences
are widely-used in probability, statistics, and many application
areas. Therefore, a characterization of asymmetries, as we give in this paper,
may be of interest outside of probabilistic forecasting.

\subsection{Structure of this paper}

The rest of this paper is structured as follows. In Section
\ref{sec:preliminaries}, we formally introduce the setting of probabilistic
forecasting and the loss functions that are discussed in this paper. In Section
\ref{sec:main-results}, we derive our main theoretical results regarding
asymmetric penalties. In Section \ref{sec:empirical-results}, we present
empirical analyses of three real data sets, as well as synthetic data. In
Section \ref{sec:hedging}, we leverage our results to shed light on hedging
proper loss functions under distribution shift. In Section \ref{sec:discussion},
we conclude with a discussion of our findings, some assorted additional results,
and ideas for future work.

\subsection{Related literature}

References to the broader literature on proper losses will be given in the next
section, and here we restrict our attention to literature more narrowly adjacent
to the focus of our paper. We are of course not the first to consider the
operating characteristics of proper losses, and how this might influence the
choice of which loss to use. This has been studied in probabilistic
classification by \citet{buja2005loss}, and in the context of
eliciting forecasts from experts by \citet{carvalho2016overview,
  merkle2013choosing, bickel2007some, machete2013contrasting}.
\citet{wheatcroft2021evaluating} articulated well the need for studying what
forecasting errors different losses favor, and left it for future work. In the
context of point forecasts, the same has been emphasized by
\citet{ehm2016quantiles} and \citet{patton2020comparing}. The work of
\citet{machete2013contrasting} adopts a broadly similar approach to ours, though
in a different setting, and reaches different conclusions. Recently,
\citet{resin2024shift} obtained a decomposition of the expected CRP loss into
shift and dispersion components. Our work may be viewed as complementing theirs,
toward an understanding of the behavior of CRP loss under differences in
location and scale between the forecast and target distribution.

\section{Preliminaries}
\label{sec:preliminaries}

We now introduce the formal setting which we study in this paper. Suppose that
$Y$ is a random variable taking values in an outcome space $\cY$, according to
some probability distribution $G$. We call $Y$ the \emph{target variable} and
$G$ the \emph{target distribution}. Letting $\cF$ denote a set of probability
distributions on $\cY$, to measure how close a \emph{forecast} $F\in\cF$ is to
the target distribution, we define a \emph{loss function} $\ell:\cF\times\cY
\to [-\infty,\infty]$, such that $\ell(F,y)$ is the loss incurred for a
forecast $F$ of the observation $y$. The forecast $F$ may be given in the form
of a measure, density, cumulative distribution function, quantiles, or
otherwise; different loss functions are defined in terms of different forms of
inputs, in this regard. To keep the setting general and clear, we use the
general word \emph{distribution} to describe the forecast $F$ and target $G$ so
as not to prefer any particular form. We assume that low \emph{expected loss}
indicates a good forecast, the expected loss defined as
\[
\ell(F,G) = \E\ell(F,Y), \quad \text{for $Y \sim G$}.
\]
(It will be clear from the context whether $\ell$ denotes the loss or expected
loss.) A loss is said to be \emph{proper}, if
\[
\ell(F,G) \geq \ell(G,G), \quad \text{for any $F,G$}.
\]
In other words, if the target distribution itself minimizes the expected loss. A
proper loss is said to be \emph{strictly proper} if strict inequality holds in
the above whenever $F \not= G$. Any proper loss induces a divergence
\[
d(F,G) = \ell(F,G)-\ell(G,G),
\]
which is nonnegative, and vanishes when the forecast $F$ equals the target
distribution $G$. We emphasize that in these definitions, and in our study in
general, the forecast $F$ is not treated as random, but as a fixed probability
distribution.

\subsection{Loss functions}

We now introduce the loss functions studied in this paper. We refer the reader
to \citet{gneiting2007strictly} for a more comprehensive review of loss
functions in probabilistic forecasting.

\textbf{Logarithmic loss.} When the outcome space $\cY$ is a convex subset of
the finite-dimensional Euclidean space $\R^d$, and $F,G$ are absolutely
continuous with respect to the Lebesgue measure, we let $f,g$ denote their
respective densities. One of the earliest examples of proper losses, logarithmic
loss \citep{good1952rational} is defined as the negative log-likelihood:
\[
\ell(F,y) = -\log f(y).
\]
Log loss is strictly proper if restricted to the family of densities with finite
Shannon entropy (i.e., $\ell(F,F)<\infty$), and it induces the well-known
Kullback--Leibler (KL) divergence \citep{kullback1951information}:
\[
d(F,G) = \int_\cY g(y) \log\frac{g(y)}{f(y)} dy
= \E \bigg[\log\frac{g(Y)}{f(Y)}\bigg] = \KL(g\|f),
\]
recalling that $Y$ is a random variable with distribution $G$. Logarithmic loss
may be defined on a general outcome space $\cY$, given that $F,G$ are both
absolutely continuous with respect to a common measure on $\cY$, but the results
in this paper pertain to Euclidean outcome spaces, as defined above. Log loss is
commonly used in epidemiology \citep{reich2019collaborative} and in eliciting
probabilities from experts \citep{carvalho2016overview}.

\textbf{Quadratic loss.} In the same setting as logarithmic loss, quadratic loss
\citep{brier1950verification} is affine in the likelihood:
\[
\ell(F,y) = -2f(y) + \int_\cY f(x)^2 dx,
\]
assuming the density $f$ is square integrable. The integral term constitutes a
penalty for low entropy. Quadratic loss is strictly proper. It is worth noting
that if the integral penalty is omitted, then the loss becomes improper. The
induced divergence is the squared $L^2$ distance between the densities:
\[
d(F,G) = \int_\cY (f(y) - g(y))^2 dy.
\]
Quadratic loss is sometimes used in the setting of forecast aggregation
\citep{hora2015calibration}. More recently, it has been used to evaluate
estimated distributions of network traffic features
\citep{dietmuller2024fitnets}.

\textbf{Spherical loss.} In the same setting as quadratic losses, spherical loss
\citep{thornton1965belief} is now linear in the likelihood:
\[
\ell(F,y) = -f(y) \bigg(\int_\cY f(x)^2 dx\bigg)^{-\frac{1}{2}},
\]
assuming the density is square integrable. Similar to quadratic loss, the
integral term here constitutes a penalty for low entropy. Spherical loss is
strictly proper, and the induced divergence is
\[
d(F,G) = \bigg(\int_\cY g(y)^2 dy\bigg)^{\frac{1}{2}} - \bigg(\int_\cY f(y)^2
dy\bigg)^{-\frac{1}{2}} \int_\cY f(y)g(y) dy.
\]
Spherical loss is a special case of pseudospherical loss
\citep{good1971comment}. This has recently gained popularity in machine
learning, with \citet{lantao2021pseudo} deriving an efficient approach for
fitting for energy-based models via minimization of pseudospherical loss
functions. \citet{lee2022pseudo} proposed using a pseudospherical loss in
knowledge distillation, whose goal is to transfer knowledge from a larger to a
smaller model.

\textbf{Continuous ranked probability (CRP) loss.} Unlike logarithmic,
quadratic, and spherical losses, we now no longer require the existence of
densities. In the case that the outcome space $\cY$ is a convex subset of the
real line $\R$, we identify $F$ with a cumulative distribution
function. Continuous ranked probability (CRP) loss \citep{matheson1976scoring}
is defined by
\[
\ell(F,y) = \int_\cY (F(x) - \mathbb{I}\{y \leq x\})^2 dx.
\]
where $\mathbb{I}\{y \leq x\}$ is equal to 1 if $y \leq x$ and 0 otherwise. This
has the alternative representation \citep{baringhaus2004new, szekely2005new}:
\[
\ell(F,y) = \E|X-X'|\bigg(\frac{\E |X-y|}{\E |X-X'|} - \frac{1}{2}\bigg),
\]
where $X,X' \sim F$ are independent. The multiplier $\E|X-X'|$ outside of the
parentheses serves as a penalty for high entropy (in contrast to quadratic and
spherical losses), as the term inside the parentheses is scale invariant. CRP
loss is strictly proper if restricted to the family of distributions with finite
first moment. It induces the Cramér divergence \citep{cramer1928composition}:
\[
d(F,G) = \int_\cY (F(y) - G(y))^2dy.
\]
These properties of CRP loss---that it does not require a forecast to have
a density, and has a representation in terms of expectations---make it quite
popular. In particular, CRP loss is widely used in atmospheric sciences
\citep{gneiting2005calibrated, scheuerer2015probabilistic, rasp2018neural,
  kochkov2024neural, clement2024distributional}, and features in isotonic
distributional regression \citep{henzi2021isotonic}. In addition, CRP loss has
another alternative representation as an integrated loss in terms of the
quantile function $F^{-1}$ \citep{laio2007verification}. This provides a close
connection to (weighted) interval score, which has recently become popular in
epidemic forecasting \citep{bracher2021evaluating}.

\textbf{Threshold-weighted CRP loss.} A threshold-weighted version of CRP loss
\citep{matheson1976scoring, gneiting2011comparing} generalizes CRP loss by
introducing a nonnegative integrable weight function $w:\cY \to [0,\infty)$, via
\[
\ell(F,y) = \int_\cY w(x) (F(x) - \mathbb{I}\{y \geq x\})^2 dx.
\]
This has the alternative representation in terms of expectations:
\[
\ell(F,y) = \E|v(X)-v(X')|\bigg(\frac{\E |v(X)-v(y)|}{\E |v(X)-v(X')|} -
\frac{1}{2}\bigg),
\]
where $X,X'\sim F$ are independent, and where $v$ is any antiderivative of $w$
(i.e., $v'=w$). Threshold-weighted CRP loss is strictly proper if restricted to
the family of distributions with finite first moment, and restricted to a
strictly positive weight function. It induces the weighted Cramér divergence:
\[
d(F,G) = \int_\cY w(y) (F(y) - G(y))^2 dy.
\]
Threshold-weighted CRP loss is used in the forecasting of high-impact events,
for example, in meteorological sciences \citep{allen2023weighted,
  taillardat2023evaluating}. We refer to the case where the weight function is
a power, $w(y)=y^\alpha$ for $\alpha\in\R$, as a power-weighted CRP loss.

\textbf{Energy loss.} Generalizing CRP in a different direction is the energy
loss \citep{gneiting2007strictly}, which extends the expectation formulation of
CRP to a multidimensional outcome space $\cY\subseteq\R^d$, using a parameter
$\beta \in (0,2)$, via
\[
\ell(F,y) = \E \|X-X'\|_2^\beta\bigg(\frac{\E \|X-y\|_2^\beta}{\E
  \|X-X'\|_2^\beta} - \frac{1}{2}\bigg),
\]
where $X,X' \sim F$ are independent, and $\|\cdot\|_2$ denotes the Euclidean
norm. Note that CRP loss is given by the special case where $d=\beta=1$. As with
CRP loss, the multiplier outside of the parentheses may be seen as a penalty
for high entropy. Energy loss is strictly proper if restricted to the family of
distributions with finite $\beta$-moment. The induced divergence is the
generalized energy distance:
\[
d(F,G) = \E\|X-Y\|_2^\beta - \frac{1}{2}\bigg(\E\|X-X'\|_2^\beta + \E
  \|Y-Y'\|_2^\beta\bigg),
\]
which is related to energy statistics \citep{szekely2013energy}. Here,
$X,X',Y,Y'$ are independent random variables having distributions $F,F,G,G$,
respectively. As a generalization of CRP loss, energy loss is employed in many
of the same application areas and also admits weighted versions that are used to
evaluate, for example, forecasts of extreme meteorological events
\citep{allen2023evaluating}.

\textbf{Dawid--Sebastiani (DS) loss.} Introduced in \citet{dawid1999coherent}
originally for non-forecasting purposes, the Dawid--Sebastiani (DS) loss is
given for a real-valued outcome space $\cY \subseteq \R$ by
\[
\ell(F, y) = \log\Var(X) + \frac{(y - \E X)^2}{\Var(X)},
\]
where $X \sim F$, assumed to have finite variance. It is, up to additive and
multiplicative constants, the logarithmic loss when the input is normally
distributed with the same mean and variance as the forecast $F$. Being dependent
only on the forecast mean and variance, it is proper but not strictly
proper. The induced divergence is
\[
d(F, G) = \log\frac{\Var(X)}{\Var(Y)}+ \frac{\Var(Y) - \Var(X) + (\E Y - \E
  X)^2}{\Var(X)},
\]
where recall $Y \sim G$. This identifies with the divergence induced by
logarithmic loss when both the forecast and target distributions are normal,
with mean and variance equal to those of $F$ and $G$, respectively. DS loss (or
its multivariate extension) have been used as approximations to logarithmic and
energy losses, in meteorology \citep{feldmann2015spatial,
scheuerer2015variogram}, epidemiology \citep{meyer2014power} and economics
\citep{lerch2017forecaster}.

\subsection{More background and commentary}

A loss $\ell$ is said to be \emph{local} if $\ell(F,y)$ depends only on the
observed target $y$, and the density evaluated at the target $f(y)$. Up to
affine transformations, logarithmic loss is the only proper loss that is local,
for outcome spaces with more than two elements
\citep{bernardo1979expected,parry2012proper}. Some authors argue that locality
is a desirable property on philosophical grounds---a local loss depends only on
target values that were realized. On the other hand, proper losses that
naturally derive from decision making tasks are often non-local
\citep{dawid2007geometry}. In some situations, it can be demanding or even
intractable to compute a function of the entire distribution. Indeed,
\citet{shao2024language} state that in language modeling the logarithmic loss
is almost exclusively used, due in part to its desirable characteristics and in
part to its tractability given the notoriously large sample space.

The log loss can be infinite if the forecast fails to attribute positive
density to the observed outcome. Proponents of this property argue it elicits
careful assessment of all possible events, no matter how unlikely they are
\citep{benedetti2010scoring}, while opponents point to the difficulty and
potential arbitrariness in assigning very small probabilities simply to avoid an
infinite loss \citep{selten1998axiomatic}.

In weather and climate applications, it is common to form a probabilistic
ensemble forecast out of point predictions from numerical simulation models with
different initial conditions \citep{leutbecher2008ensemble}. CRP loss is
widespread in those fields, with extensive use that can be partially explained
by the convenience it provides: when CRP is represented in its alternative
expectation-based form, the forecast $F$ can be readily replaced by
sample averages over the members of an ensemble, or the output of a Markov
chain Monte Carlo procedure \citep{gneiting2007probabilistic,
allen2023weighted}. Compared to the logarithmic loss, a further reason for using
CRP is that it retains propriety when straightforward and natural weighting
schemes are applied to emphasize certain events \citep{matheson1976scoring,
gneiting2011comparing}, whereas the logarithmic loss does not
\citep{amisano2007comparing}, and weighting schemes that maintain propriety for
log loss can be less intuitive.

\section{Main results}
\label{sec:main-results}

We now turn to proving the main results, dividing our presentation on scale
families (Section \ref{sec:main-scale}), location families (Section
\ref{sec:main-location}), and  exponential families (Section
\ref{sec:main-exponential}).

\subsection{Scale families}
\label{sec:main-scale}

For a scale family $\{G_\sigma:\sigma>0\}$, in order to study asymmetry inherent
to a loss $\ell$ (whether $\ell$ favors overestimating or underestimating the
scale $\sigma$), we impose two conditions on the associated divergence
$d$. First we assume $d$ is \emph{symmetric}, which means it is invariant to
the order of its arguments, that is,
\[
d(F,G) = d(G,F), \quad \text{for all $F,G$}.
\]
It is clear from their definitions in Section \ref{sec:preliminaries} that the
divergences induced by quadratic, CRP, threshold-weighted CRP, and energy losses
are symmetric, whereas those induced by spherical, DS, and logarithmic losses
are not. Spherical and DS losses deviate from symmetry in ways that Theorem
\ref{thm:scale-family}---proved in this subsection---is able to accommodate,
while logarithmic loss does not and is left for later treatment.

The second requirement is that there exists a function $h: (0,\infty) \to
(0,\infty)$ such that
\[
d(F_\sigma,G_\tau) = h(\tau) d(F_{\sigma/\tau},G_1), \quad \text{for all
  $\sigma,\tau>0$},
\]
and for all scale families $\{F_\sigma:\sigma>0\}, \{G_\sigma:\sigma>0\}$. This
says that the forecast and target distributions can be jointly rescaled by the
same factor, at the cost of a multiplicative element applied to the
divergence. We call divergences which satisfy this condition \emph{rescalable},
and we call $h$ the \emph{scaling} function. Note that $h$ must always satisfy
$h(1)=1$.

Interestingly, the only possible continuous scaling functions $h$ are real
powers: $h(\sigma)=\sigma^\gamma$, for some $\gamma\in\R$. Indeed, if we equate
rescaling by $\sigma\tau$ to rescaling by $\sigma$ and then $\tau$, we
obtain the identity $h(\sigma\tau)=h(\sigma)h(\tau)$, for $\sigma,\tau>0$. This
is often called Pexider's equation, and subject to $h(1)=1$, the only continuous
functions which satisfy this equation are real powers
\citep{small2007functional}.

The following lemma summarizes the rescalability properties of the loss
functions discussed in this paper.

\medskip
\begin{lemma}
\label{lem:rescalability}
The following holds.
\begin{itemize}
\item CRP loss induces a rescalable divergence, with $h(\sigma)=\sigma$.
\item Energy loss induces a rescalable divergence, with
  $h(\sigma)=\sigma^\beta$.
\item Quadratic loss induces a rescalable divergence, with
  $h(\sigma)=1/\sigma$.
\item Logarithmic and DS losses induce rescalable divergences, with
  $h(\sigma)=1$.
\item Power-weighted CRP loss ($w(y) = y^\alpha$) induces a rescalable
  divergence, with $h(\sigma)=\sigma^{\alpha+1}$.
\end{itemize}
\end{lemma}
\smallskip

The proof in each case is by change of variables and is omitted.

From the perspective of Theorem \ref{thm:scale-family}, what shall matter is
whether the scaling $h$ is an increasing, decreasing or constant function. To be
clear, we say that $h$ is increasing if $h(x)<h(y)$ for all $0<x<y$, and
decreasing if $-h$ is increasing. The divergences induced by CRP, energy, and
power-weighted CRP (with $\alpha>-1$) losses have an increasing scaling
function, and the ones induced by quadratic and power-weighted CRP (with
$\alpha<-1$) losses have a decreasing scaling function. The divergences induced
by DS, logarithmic and power-weighted CRP (when $\alpha=-1$) losses have
constant scaling, so we may call them \emph{scale invariant}.

A key step for the proof of Theorem \ref{thm:scale-family} is given in the next
lemma, which characterizes whether a loss function favors underestimating the
scale, overestimating it, or neither, based on symmetry and rescalability of the
induced divergence.

\medskip
\begin{lemma}
\label{lem:scale-asymmetry}
Let $\{G_\sigma:\sigma>0\}$ be a scale family, $\ell$ a loss function, and
$d$ the induced divergence. If $d$ is symmetric and rescalable with scaling
function $h$, then the following holds for $G=G_1$ and all $\sigma>1$.
\begin{itemize}
\item If $h$ is increasing, then $\ell(G_{\sigma}, G) > \ell(G_{1/\sigma}, G)$.
\item If $h$ is decreasing, then $\ell(G_{\sigma}, G) < \ell(G_{1/\sigma}, G)$.
\item If $h$ is constant, then $\ell(G_{\sigma}, G) = \ell(G_{1/\sigma}, G)$.
\end{itemize}
\end{lemma}

\begin{proof}
Using symmetry then rescalability, we compute:
\[
d(G_{\sigma}, G) = d(G, G_{\sigma}) = h(\sigma) d(G_{1/\sigma}, G).
\]
Suppose that the scaling $h$ is increasing, so that $h(\sigma)>h(1)=1$. It
follows that $d(G_\sigma, G) > d(G_{1/\sigma}, G)$, from which we conclude
$\ell(G_{\sigma}, G) - \ell(G_{1/\sigma}, G) = d(G_{\sigma}, G) -
d(G_{1/\sigma}, G) > 0$. The cases where the scaling $h$ is decreasing or
constant follow similarly.
\end{proof}

The proof of Theorem \ref{thm:scale-family} is now just a matter of putting
together Lemmas \ref{lem:rescalability} and \ref{lem:scale-asymmetry}, with some
additional work along the same lines to accommodate spherical and DS losses.

\begin{proof}[Proof of Theorem \ref{thm:scale-family}]
CRP and energy losses are symmetric and by Lemma \ref{lem:rescalability}
they are rescalable with increasing scaling. Hence Lemma
\ref{lem:scale-asymmetry} implies $\ell(G_\sigma, G)>\ell(G_{1/\sigma},G)$
as required. Meanwhile, quadratic loss is symmetric and by Lemma
\ref{lem:rescalability} it is rescalable with decreasing scaling, so by
Lemma \ref{lem:scale-asymmetry} we conclude that $\ell(G_{\sigma}, G) <
\ell(G_{1/\sigma}, G)$, as required.

Suppose now that $\ell$ is the spherical loss. The induced divergence $d$ is
neither symmetric nor rescalable. However, when the forecast and target
distribution are members of the same scale family, we have the following two
properties. First, a condition similar to rescalability holds, with decreasing
scaling:
\[
d(G_\sigma, G_\tau) = \frac{1}{\sqrt{\tau}} d(G_{\sigma/\tau}, G).
\]
This may be shown using a change of variables. Second, a property related to
symmetry holds:
\[
d(G_{\sigma}, G_\tau) = \sqrt{\frac{\sigma}{\tau}}d(G_\tau, G_\sigma).
\]
Interestingly, the effects of the two multiplicative factors above cancel each
other, resulting in $d(G_\sigma, G) = \sqrt{\sigma} d(G,G_\sigma) =
d(G_{1/\sigma},G)$, and thus $\ell(G_\sigma, G)=\ell(G_{1/\sigma},G)$, as
required.

Finally, for $\ell$ being DS loss, and $d$ the associated divergence, we may
check by direct differentiation that
\[
d(G_{1/\sigma}, G) - d(G_{\sigma}, G) = \bigg(\sigma^2 -
  \frac{1}{\sigma^2}\bigg) - 2\log\sigma > 0.
\]
Consequently, $\ell(G_{\sigma}, G) < \ell(G_{1/\sigma}, G)$, as required. This
completes the proof.
\end{proof}

\subsection{Location families}
\label{sec:main-location}

The argument for a location family $\{G_\mu:\mu\in\R\}$ is broadly similar to
that given in the previous subsection for a scale family. We assume that the
divergence $d$ induced by the given loss $\ell$ is symmetric as well as
\emph{translation invariant}:
\[
d(F_\mu, G_\nu)=d(F_{\mu-\nu}, G_0), \quad \text{for all $\mu,\nu \in \R$},
\]
and for all location families $\{F_\mu:\mu\in\R\}, \{G_\mu:\mu\in\R\}$. It is
clear from their definitions in Section \ref{sec:preliminaries} that the
divergences induced by quadratic, CRP, energy, spherical, DS, and logarithmic
losses are translation invariant.

The next lemma is similar to Lemma \ref{lem:scale-asymmetry} and is key to
the proof of Theorem \ref{thm:location-family}.

\medskip
\begin{lemma}
\label{lem:translation-invariant}
Let $\{G_\mu:\mu\in\R\}$ be a location family, $\ell$ a loss function, and $d$
the induced divergence. If $d$ is symmetric and translation invariant, then
$\ell(G_\mu, G) = \ell(G_{-\mu},G)$ for $G=G_0$ and all $\mu\in\R$.
\end{lemma}

\begin{proof}
Using symmetry then translation invariance, $d(G_\mu, G) = d(G, G_\mu) =
d(G_{-\mu}, G)$. We conclude $\ell(G_\mu, G) - \ell(G_{-\mu}, G) = d(G_{\mu}, G)
- d(G_{-\mu}, G) = 0$.
\end{proof}

We now prove Theorem \ref{thm:location-family}.

\begin{proof}[Proof of Theorem \ref{thm:location-family}]
For quadratic, CRP, and energy losses, the associated divergences are symmetric
and translation invariant, so application of Lemma
\ref{lem:translation-invariant}
leads to the desired results.

The divergences induced by spherical and DS losses are translation invariant,
and furthermore they are symmetric when restricted to the case where the
forecast and target distribution both belong to the same location
family. Therefore the same argument as in Lemma \ref{lem:translation-invariant}
gives the desired result for spherical and DS losses.

For logarithmic loss, let $g$ be the Lebesgue density of $G$. Recall $g$ is
itself assumed to be symmetric, and by translation  invariance we may assume,
without loss of generality, that $g$ is symmetric around zero, i.e.,
$g(y)=g(-y)$. By a change of variables, followed by symmetry,
\[
d(G_\mu, G) = \int g(y)\log\frac{g(y)}{g(y-\mu)} dy = \int
g(-y)\log\frac{g(-y)}{g(-y-\mu)} dy = \int g(y)\log\frac{g(y)}{g(y+\mu)}dy =
d(G_{-\mu}, G).
\]
We conclude that $\ell(G_\mu, G)=\ell(G_{-\mu}, G)$, as required.
\end{proof}

\subsection{Exponential families}
\label{sec:main-exponential}

For logarithmic loss, the arguments given in Section \ref{sec:main-scale}, which
are based on symmetry of the divergence, do not apply. Indeed, the KL
divergence, which is the divergence associated with log loss and is not
symmetric in its arguments, can behave in subtle ways for scale families. For
the Cauchy scale family, log loss is actually symmetric in its penalty for
$\sigma$ versus $1/\sigma$ (this follows from \citealp{nielsen2019kullback}),
which shows that this loss does not always favor overestimating a scale
parameter, despite what the empirical results in Figures
\ref{fig:heatmaps-normal} and \ref{fig:heatmaps-covid} might seem to suggest at
face value.

The approach we take in this subsection is to characterize the asymmetries
inherent to logarithmic loss within minimal single-parameter exponential
families. Let $\{p_\eta:\eta>0\}$ be a minimal exponential family of densities,
where
\[
p_\eta(x) = h(x) e^{\eta T(x)-A(\eta)},
\]
and the sufficient statistic $T$ is not almost everywhere constant. Below we
establish conditions for the logarithmic loss to favor forecasting
$p_{\theta\eta}$ over $p_{\eta/\theta}$, the opposite, or neither, for any
$\theta>1$, when the target distribution is $p_\eta$. These conditions are
described in terms of the second derivative of the log-partition function
$A(\eta)$; note the second derivative always exists and is nonnegative
\citep{keener2010theoretical}. In particular, logarithmic loss will favor
underestimating the natural parameter if the acceleration $A''(\eta)$ dominates
a rate $1/\eta^3$, in a particular sense to be made precise. On the other
hand, logarithmic loss will favor overestimating the natural parameter for
exponential families where the acceleration $A''(\eta)$ is dominated by the rate
$1/\eta^3$. When the natural parameter space is the entire real line (instead of
the positive reals), we derive similar conditions defined in terms of
$A''(\eta)$ versus $A''(-\eta)$. Our results are summarized in the following.

\medskip
\begin{lemma}
\label{lem:acceleration}
Let $\{p_\eta:\eta\in\Omega\}$ be a minimal single-parameter exponential family,
with the log-partition function $A(\eta)$. Denote by $\ell$ the logarithmic
loss. If $\Omega=(0,\infty)$, then the following holds for all $\eta\in\Omega$
and all $\theta>1$.
\begin{enumerate}
\item If $u^3A''(u)$ is increasing in $u$, then $\ell(p_{\theta\eta}, p_\eta) >
  \ell(p_{\eta/\theta}, p_\eta)$.
\item If $u^3A''(u)$ is decreasing in $u$, then $\ell(p_{\theta\eta}, p_\eta) <
  \ell(p_{\eta/\theta}, p_\eta)$.
\item If $u^3A''(u)$ is constant in $u$, then $\ell(p_{\theta\eta}, p_\eta) =
  \ell(p_{\eta/\theta}, p_\eta)$.
\end{enumerate}
If $\Omega=\R$, then the following holds for all $\eta\in\Omega$ and all
$\theta>0$.
\begin{enumerate}
\item If $A''(u)-A''(-u)$ is increasing in $u$, then $\ell(p_{\eta+\theta},
  p_\eta) > \ell(p_{\eta-\theta}, p_\eta)$.
\item If $A''(u)-A''(-u)$ is decreasing in $u$, then $\ell(p_{\eta+\theta},
  p_\eta) < \ell(p_{\eta-\theta}, p_\eta)$.
\item If $A''(u)-A''(-u)$ is constant in $u$, then $\ell(p_{\eta+\theta}, p_\eta)
  = \ell(p_{\eta-\theta}, p_\eta)$.
\end{enumerate}
\end{lemma}

\begin{proof}
Suppose $\Omega=(0,\infty)$. Using scale invariance, we may assume without loss
of generality that $\eta=1$. Denoting $p=p_1$, note that the difference in
divergence can be written as
\[
\ell(p_{\theta}, p)-\ell(p_{1/\theta}, p) = d(p_{\theta}, p)-d(p_{1/\theta}, p)
= \E \log\frac {p_{1/\theta}(Y)}{p_{\theta}(Y)} = A(\theta) - A(1/\theta) -
(\theta-1/\theta) \E T(Y),
\]
where recall $Y$ is a random variable with density $p$. A standard exponential
family identity states that $\E T(Y)=A'(1)$, hence we may write
\[
\ell(p_{\theta}, p)-\ell(p_{1/\theta}, p) = d_A(\theta, 1) - d_A(1/\theta, 1),
\]
where $d_A(y,x)=A(y)-A(x)-A'(x)(y-x)$ is the Bregman divergence associated with
the log-partition function $A$. Assume $u^3 A''(u)$ is increasing. Applying the
fundamental theorem of calculus twice, we bound one of the Bregman divergences
from below,
\[
d_A(\theta, 1) = \int_1^{\theta} \int_1^v A''(u) du dv =
\int_1^{\theta} \int_1^v \frac{u^3 A''(u)}{u^3} du dv >
A''(1) \int_1^{\theta} \int_1^v \frac{du dv}{u^3}.
\]
Applying the same argument again, we bound the other Bregman divergence from
above,
\[
d_A(1/\theta, 1) = \int_{1/\theta}^1 \int_{v}^1 A''(u) du dv =
\int_{1/\theta}^1 \int_{v}^1 \frac{u^3 A''(u)}{u^3} du dv <
A''(1) \int_{1/\theta}^1 \int_{v}^1 \frac{du dv}{u^3}.
\]
Now, the statement $\ell(p_{\theta}, p) > \ell(p_{1/\theta}, p)$ will have been
established as soon as we notice that the two integrals identify,
\[
\int_1^{\theta} \int_1^v \frac{du dv}{u^3} =
\frac{(\theta-1)^2}{2\theta} =
\int_{1/\theta}^1 \int_{v}^1 \frac{du dv}{u^3}.
\]
The statements for $u^3A''(u)$ decreasing and constant, as well as the
statements for the case that $\Omega=\R$, are proven in a similar way, and we
omit the details.
\end{proof}

We now turn to proving Theorem \ref{thm:exponential-family}. Before we do that,
let us define the exponential families covered by the theorem.

\begin{itemize}
\item The generalized gamma scale family has density
  $g_{\sigma}(x)=\Gamma(k/\gamma)^{-1}\sigma^{-k}\gamma
  x^{k-1}e^{-(x/\sigma)^\gamma}$, for $x,\sigma,\gamma,k>0$, and
  where $\Gamma$ is the gamma function. Here the natural parameter is
  $\eta=\sigma^{-\gamma}\in(0,\infty)$ and the log-partition function is
  $A(\eta)=-(k/\gamma)\log\eta$. Special cases are the gamma scale family
  ($\gamma=1$), exponential scale family ($\gamma=1$ and $k=1$) and the Weibull
  scale family ($k=\gamma$). If the generalized gamma scale family is extended
  by symmetry to $x\in \R$, we obtain the Laplace scale family ($\gamma=1$ and
  $k=1$) and the normal scale family ($\gamma=2$ and $k=1$).

\item The log-normal log-scale family has density
  $g_\sigma(x)=(x\sigma\sqrt{2\pi})^{-1} e^{-(\log{x}-\mu)^2/(2\sigma^2)}$,
  where $x,\sigma>0$ and $\mu\in\R$. The natural parameter is
  $\eta=\sigma^{-2}\in(0,\infty)$ and the log-partition function is
  $A(\eta)=-(1/2)\log\eta$.

\item The inverse gamma scale family has density
  $g_\sigma(x)=\Gamma(k)^{-1}\sigma^kx^{-k-1}e^{-\sigma/x}$, for
  $x,\sigma,k>0$. Here the natural parameter is $\eta=\sigma\in(0,\infty)$ and
  the log-partition function is $A(\eta)=-k\log\eta$.

\item The generalized gamma shape family has density
  $g_k(x)=\Gamma(k/\gamma)^{-1}\sigma^{-k}\gamma x^{k-1}e^{-(x/\sigma)^\gamma}$,
  for $x,\sigma,\gamma,k>0$. The natural parameter is $\eta=k\in(0,\infty)$ and
  the log-partition function is $A(\eta)=(\log\sigma)\eta +
  \log\Gamma(\eta/\gamma)$. A special case is the gamma shape family
  ($\gamma=1$).

\item The Pareto shape family has density $g_k(x)=km^k/x^k$, for $x\geq m$ and
  $k,m>0$. Here the natural parameter is $\eta=k\in(0,\infty)$ and the
  log-partition function is $A(\eta)=-\log\eta-(\log m)\eta$.

\item The inverse Gaussian shape family has density $g_k(x)=\sqrt{k/2\pi x^3}
  e^{-k(x-\mu)^2/(2\mu^2x)}$, for $x,k,\mu>0$. The natural parameter is
  $\eta=k\in(0,\infty)$ and the log-partition function is
  $A(\eta)=-(1/2)\log\eta$.

\item The beta shape family has density
  $g_k(x)=x^{k-1}(1-x)^{\beta-1}\Gamma(k+\beta)/(\Gamma(k)\Gamma(\beta))$, for
  $x\in[0,1]$ and $k,\beta>0$. The natural parameter is $\eta=k\in(0,\infty)$
  and the log-partition function is
  $A(\eta)=-\log(\Gamma(\eta+\beta)/\Gamma(\eta))$.

\item The Poisson rate family has density
  $g_\lambda(k)=\lambda^ke^{-\lambda}/k!$, for $k\geq 0$ integer and
  $\lambda>0$. Here the natural parameter is $\eta=\log\lambda\in\R$ and the
  log-partition function is $A(\eta)=e^\eta$.
\end{itemize}

\begin{proof}[Proof of Theorem \ref{thm:exponential-family}]
For families with a log-partition function of the form
$A(\eta)=-a\log\eta+b\eta$, where $a>0$ and $b\in \R$, the acceleration is
$A''(\eta)=a/\eta^2$ which means that $u^3 A''(u)$ is increasing in $u$. Hence,
by Lemma \ref{lem:acceleration}, we have $\ell(p_{\theta\eta},p_\eta) >
\ell(p_{\eta/\theta},p_\eta)$ for all $\eta>0$ and all $\theta>1$. Consequently,
for the generalized gamma scale family, as the scale parameter $\sigma$ is
inversely proportional to the natural parameter $\eta=1/\sigma^\gamma$, we have,
writing \smash{$g_{\theta^{1/\gamma}}$} for the density of $G_\theta$,
\[
\ell(G_{1/\theta}, G) =
\ell\big( g_{\theta^{-1/\gamma}}, g \big) =
\ell(p_\theta,p) > \ell(p_{1/\theta},p) =
\ell\big( g_{\theta^{1/\gamma}}, g \big) =
\ell(G_\theta, G),
\]
for all $\theta>1$. As special cases, we obtain the required
results for gamma, exponential and Weibull scale families. By symmetrization,
the same applies to the Laplace and normal scale families as well, and analogous
arguments hold for log-normal log-scale family, inverse gamma scale family, and
the Pareto and inverse Gaussian shape families.

For the generalized gamma shape family, the acceleration is
$A''(\eta)=\gamma^{-2}\psi'(\eta/\gamma)$, where $\psi$ is the digamma
function. It suffices to show that $f(x)= x^3 \psi'(x)$ is an increasing
function of $x$. From the series representation \citep{arfken2011mathematical}
\[
\psi'(x)=\sum_{n=0}^\infty \frac{1}{(x+n)^2}  \quad \text{and} \quad
\psi''(x)=-\sum_{n=0}^\infty \frac{2}{(x+n)^3},
\]
valid for all $x>0$, we differentiate:
\[
f'(x)=3x^2\psi'(x)+x^3\psi''(x) = x^2\sum_{n=0}^\infty \frac{x+3n}{(x+n)^3} > 0.
\]
Therefore, $f$ is increasing, consequently $u^3A''(u)$ is increasing in $u$, and
by Lemma \ref{lem:acceleration} we obtain the desired result. As a special case
we obtain the desired result for the gamma shape family. For the beta shape
family, the acceleration is $A''(\eta)=\psi'(\eta) - \psi'(\eta+\beta)$, and it
suffices to show that $f(x)=x^3(\psi'(x)-\psi'(x+\beta))$ is an increasing
function of $x$. We use the series representation to differentiate:
\[
f'(x)=\big(3x^2\psi'(x)+x^3\psi''(x)\big)-\left(3x^2\psi'(x+\beta) +
x^3\psi''(x+\beta)\right) = x^2\sum_{n=0}^\infty \bigg(\frac{x+3n}{(x+n)^3} -
\frac{x+3(\beta+n)}{(x+\beta+n)^3}\bigg).
\]
As $a\mapsto (x+3a)/(x+a)^3$ is decreasing, the derivative $f'$ is
positive, thus $f$ is increasing. Consequently, $u^3A''(u)$ is increasing in
$u$, and by Lemma \ref{lem:acceleration} we obtain the desired result. Finally,
for the Poisson rate family, the acceleration is $A''(\eta)=e^\eta$, with
$\Omega=\R$. Since $A''$ is increasing, by Lemma \ref{lem:acceleration} we
obtain the desired result, completing the proof.
\end{proof}

\begin{remark}
Using the same approach, we may obtain similar results for further exponential
families, and we omit the details. Examples of such families to which this
method is applicable:  normal location family, log-normal log-location family,
inverse Gaussian mean family, binomial family with success probability as
parameter.
\end{remark}

\section{Empirical results}
\label{sec:empirical-results}

We now discuss in greater detail the experiments presented in the introduction, and
report further results. We begin by highlighting the implications of the
asymmetries underlying proper loss functions on real data, over three domains. A
comprehensive synthetic experiment across five proper loss functions finishes
this section, which corroborates many of our previous findings and motivates
directions for future work. All of our results are reproducible from the code
available at \url{https://github.com/jv-rv/loss-asymmetries/}.

\subsection{Covid-19 mortality}

The Covid-19 Forecast Hub \citep{cramer2022us} has been used by the U.S.\
Centers for Disease Control and Prevention (CDC) to communicate information
leading to policy decisions, including as to the allocation of ventilators,
vaccines, and medical staff across geographic locations over the Covid-19
pandemic. We restrict our attention to death forecasts, as in
\citet{cramer2022eval}, and we analyze weekly predictions from dozens of
forecasters of the number of Covid-19 deaths one through four weeks ahead, for
each U.S.\ state, over the period from April 2020 to March 2023. All forecasts
in the Hub are stored in quantile format, and so we converted each forecast from
a discrete set of quantiles to a density or CDF before compute log or CRP loss,
respectively. Details are given in Appendix \ref{app:covid-converting}.

To compute the heatmap in Figure \ref{fig:heatmaps-covid}, we applied
transformations to the forecasts and the target data, as follows. First, we
estimated and applied nonparametric transformations to standardize the target
distribution, giving it approximately mean zero and variance one at every time
point and for every location; details are given in Appendix
\ref{app:covid-standardizing}. Second, we centered and scaled each input
forecast distribution in the Covid-19 Hub, giving it mean $\mu$ and variance
$\sigma$, and then we evaluated the loss (logarithmic or CRP) between the
transformed forecast distribution and standardized target value. We averaged
these loss values over all dates, all locations, and all forecasters; this was
done over a fine grid of $\mu,\sigma$ values in order to create the heatmap
(each pair $\mu,\sigma$ corresponds to a particular pixel value). Lastly, the
color scale for the losses in Figures \ref{fig:heatmaps-normal} and
\ref{fig:heatmaps-covid} (as well as Figure \ref{fig:asymmetric-laplace}) was
set nonlinearly in order to better emphasize the sublevel sets.

Figure \ref{tab:covid-ranking} reports the standardized ranking of ten
forecasters who participated in the Covid-19 Forecast Hub and received the top
scores in Figure 2 of \citet{cramer2022eval}. The standardized ranking is
computed in the same manner as that paper, as follows. For each state, date and
weeks ahead, we ranked the forecasters according to loss, then divided by the
number of forecasters which submitted a forecast for the state-date-weeks-ahead
combination. This is the standardized ranking value, between zero and one, which
we then averaged for each forecaster across states, dates, and weeks ahead. Each
column of Figure \ref{tab:covid-ranking} displays this for a different choice of
loss---log loss (left column) and CRP loss (right). To further annotate, we
identified the four forecasters with the highest standardized ranking according
to forecast variance (instead of loss) and the four lowest. For the four with
the highest standardized variance, logarithmic loss ranked each one at least as
high as CRP loss, with an average margin of 2 places. For the four with the
lowest standardized variance, CRP loss ranked each one at least as high as log
loss, with an average margin of 1.75 places.

\subsection{Retail sales}

The Makridakis Competitions are well-known in the applied forecasting community,
with the initial edition beginning over 40 years ago
\citep{makridakis1982accuracy}. The fifth edition introduced an uncertainty
track called the M5 Uncertainty Competition \citep{makridakis2022uncertainty},
in which teams were tasked with forecasting nine quantiles of future retail
sales for a myriad of products, aggregation levels, and horizons. Out of almost
900 competing teams, the 50 with best average performance, according to a
weighted and scaled version of the quantile loss, had their forecasts
published. Demand is notoriously intermittent, especially at the least
aggregated level, which posed a steep challenge to forecasters. Many of the
best-performing teams used gradient boosting \citep{friedman2001greedy} to
predict quantiles. As we did with the forecasts in the Covid-19 Forecast Hub,
we converted these quantile forecasts into a density or CDF before computing log
or CRP loss, respectively, following the same strategy as that outlined in
Appendix \ref{app:covid-converting}.

Each forecaster was then ranked among all forecasts issued for each prediction
task, and the rankings were standardized to be between zero and one. Averaging
the standardized ranks coming from logarithmic and CRP losses, as we did in the
Covid-19 experiment, yielded Figure \ref{fig:m5-ranking} in the introduction. As
we can see, log loss clearly prefers the forecasters that have higher
standardized forecast variance (from left to right, there is clear downward
movement in the rankings of highest-variance forecasters), while CRP clearly
prefers the opposite (from left to right, clear upward movement in the rankings
of lowest-variance forecasters).

\subsection{Temperature extremes}

The Coupled Model Intercomparison Project (CMIP) is a significant project
collecting the results of climate simulation models, introduced nearly 30 years
ago by \citet{meehl1997intercomparison}. Here we consider the maximum
temperatures over the boreal summer simulated by the 30 models available from
its sixth edition CMIP6 \citep{eyring2016overview}. We source the target data
from Hadex3 \citep{dunn2020development}, a data set of monthly extreme
temperatures (and other climate data) on a spatial grid, based on observations
recorded from thousands of meteorological stations across the world.

Using historical runs of each model from 1950 to 2014, we first calculated the
predicted monthly maximum temperature from the available daily data. We then
translated each model's simulated predictions to Hadex3's spatial grid via
bilinear interpolation routines from the Climate Data Operators collection
\citep{schulzweida2023cdo}. After this step, each model has 195 simulated
maximum temperatures at 3055 spatial locations (grid points). Adopting a
similar approach to previous work in the literature
\citep{thorarinsdottir2020evaluation}, we applied kernel density
estimation---with a Gaussian kernel and Scott's rule for automatic bandwidth
selection \citep{scott1992multivariate}---in order to form a forecast
distribution at each location. The same was done with the target data, in order
to form a target distribution at each location. (We ensure densities at each
location have the same support fitting exponential tails to the extreme
samples.)

\begin{figure}[htb]
\centering
\includegraphics[width=\textwidth]{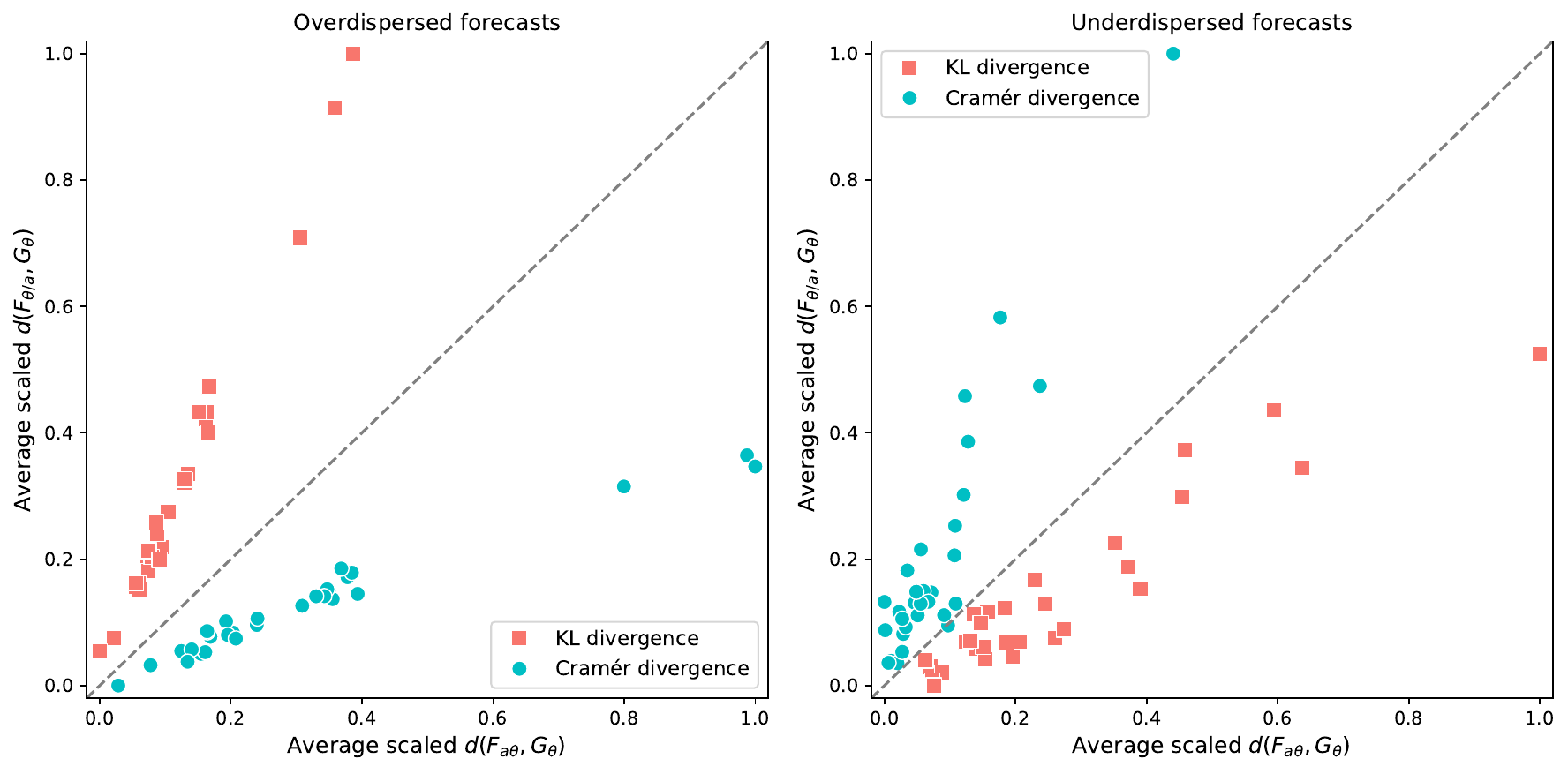}
\caption{\small Average divergence for 30 models in CMIP6, in cases where they
  are initially overdispersed (left panel) or underdispersed (right panel). For
  each forecast distribution $F_{a\theta}$, and target distribution $G_\theta$,
  we compare divergences $d(F_{a\theta}, G_\theta)$ and $d(F_{\theta/a},
  G_\theta)$ to see whether the given divergence prefers overdispersion to
  underdispersion. The average Cramér divergence is improved for all 30 models
  when initially overdispersed forecasts are made underdispersed, and the
  average KL divergence is improved for all 30 models when underdispersed
  forecasts are made overdispersed.}
\label{fig:cmip6-divergences}
\end{figure}

After computing these forecast and target densities, we then performed the
following evaluation scheme for each model. At each spatial location, we first
shifted the forecast density so that it matches the mean of the target
density---this was done to focus on differences in scale. We then estimated the
standard deviations of the forecast and target densities, and denote these
distributions by $F_{a\theta}$ and $G_\theta$, respectively. Note that if $a>1$
then the forecast distribution is overdispersed compared to the target at the
given spatial location, and if $a<1$ then it is underdispersed. Lastly,
subsetting to spatial locations with $a>1$, we computed the average divergence
$d(F_{a\theta}, G_\theta)$, as well as the average divergence $d(F_{\theta/a},
G_\theta)$. The former represents the divergence of the (original) overdispersed
forecast, while the latter represents the divergence of the underdispersed
counterpart (where the deviation in scale has been inverted).  This is displayed
on the left panel of Figure \ref{fig:cmip6-divergences}, where one point
represents one model and one divergence---either KL (associated with log loss),
or Cram{\'e}r divergence (associated with CRP loss). The right panel of Figure
\ref{fig:cmip6-divergences} displays the result of a complementary calculation:
subsetting to spatial locations with $a<1$, we compare the average divergence
$d(F_{a\theta}, G_\theta)$ and the average divergence $d(F_{\theta/a},
G_\theta)$. We can clearly see that all 30 models from CMIP6 have their
overdispersed forecasts improve in Cramér divergence when deflated and all their
underdispersed forecasts improve in KL divergence when inflated.

\subsection{Synthetic data}

The asymmetric Laplace distribution has density
\[
f(x) = \frac{p(1-p)}{\sigma} e^{-\frac{x-\mu}{\sigma} (p - \mathbb{I}
\{x \leq \mu\})},
\]
for $x \in \R$, a location parameter $\mu\in\R$, a scale parameter $\sigma>0$,
and skew parameter $p\in(0,1)$ \citep{koenker1999goodness, yu2005asymmetric}.
This reduces to the standard Laplace distribution for $p=1/2$, it is skewed to
the right for $p<1/2$, and skewed to the left for $p>1/2$. Often used in
Bayesian quantile regression \citep{yu2001bayesian}, it has been more recently
employed in probabilistic forecasting, for wind power forecasts
\citep{wang2022ensemble, wang2022deep}. The purpose of the experiment in this
subsection is twofold: it depicts behaviors one would expect given our
theoretical results, and also displays interesting phenomena outside of the
scope of our theory, pointing towards possible future work.

We computed the CRP, logarithmic, quadratic, spherical, and Dawid--Sebastiani
loss between a zero-mean unit-variance target asymmetric Laplace distribution,
and forecast distributions in the same family but with varying location and
scale. Figure \ref{fig:asymmetric-laplace} displays the results, with each row
showing a different skew parameter, and each column a different loss. The top
row corresponds to $p=0.2$ (right-skewed), the middle row to $p=0.5$
(symmetric), and the bottom row to $p=0.8$ (left-skewed). As expected, losses
penalize symmetrically on the location $\mu$ when the scale is correctly
specified, except for logarithmic loss---this loss is symmetric in $\mu$ when
the distribution itself is symmetric (middle row), but it prefers upshifted or
downshifted forecasts when the distribution is right- or left-skewed (top or
bottom rows), respectively.

\begin{figure}[htb]
\centering
\includegraphics[width=\textwidth]{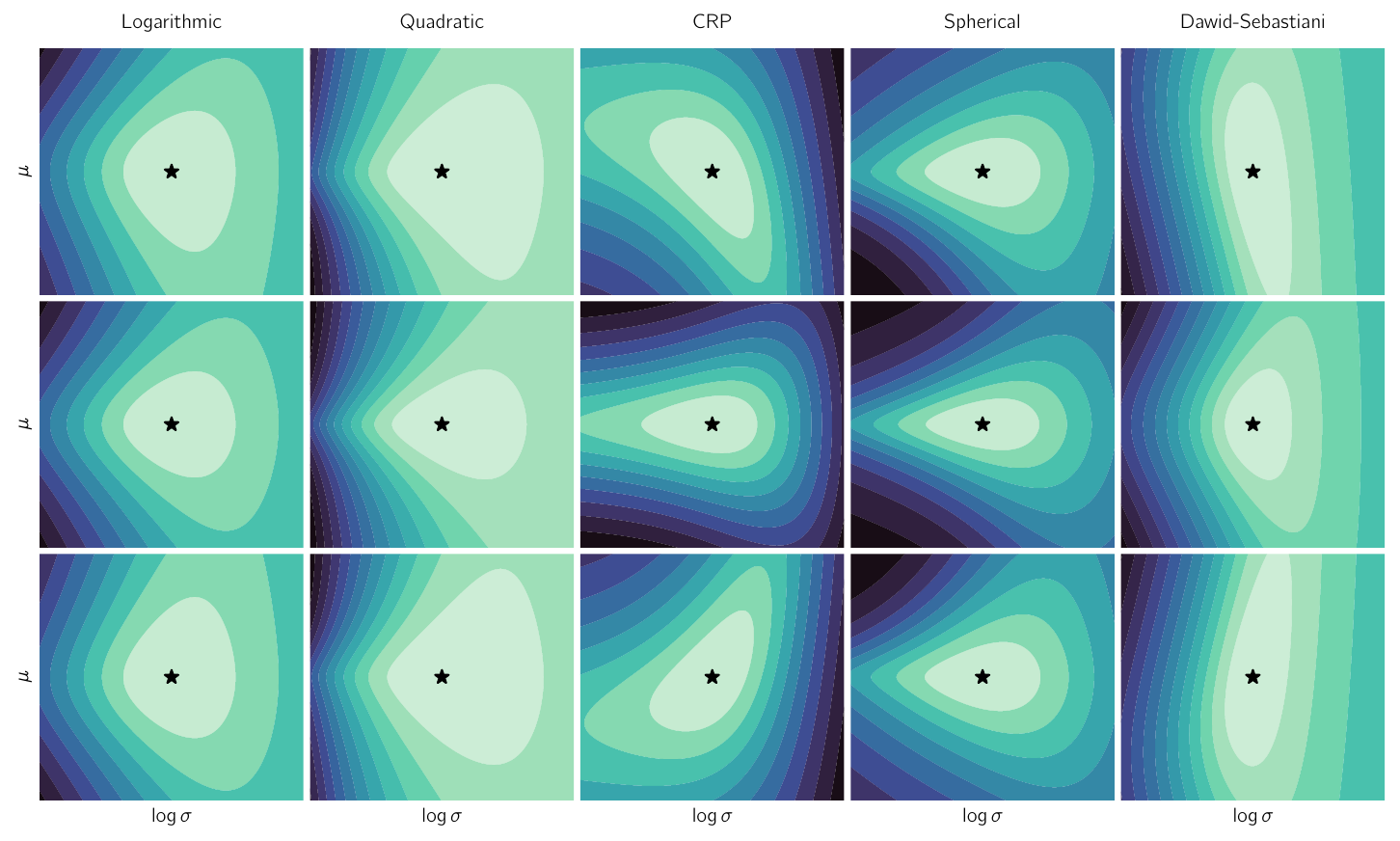}
\caption{\small Expected losses for a zero-mean unit-variance asymmetric Laplace
  target, and forecasts of the same family with varying location $\mu$ and
  scale $\sigma$. Distributions are right skewed (top panel; $p=0.2$), symmetric
  (middle; $p=1/2$), or left skewed (bottom; $p=0.8$). A lighter color
  represents a lower loss, with minimum achieved at the star.}
\label{fig:asymmetric-laplace}
\end{figure}

When the location is correctly specified, the penalties on the scale $\sigma$
all agree with what is suggested by our theory. CRP prefers underdispersed
forecasts, log loss prefers underdispersed forecasts in the case without
skewness (middle row, in which case we are in an exponential family), quadratic
and DS losses prefer overdispersed forecasts, and spherical loss is symmetric in
$\log\sigma$. Interestingly, the asymmetric penalty of log loss in the
symmetric Laplace case extends to the right- and left-skewed cases (top and
bottom rows), despite the fact that the asymmetric Laplace family is not an
exponential family. Moreover, considering the behavior of losses in $\mu,\sigma$
jointly leads to some interesting observations. For example, spherical loss,
despite being symmetric on both axes, can become asymmetric in the scale
parameter, for misspecified location. As another example, CRP seems particularly
affected by skew in the underlying distribution, with its sublevel sets
exhibiting a strong tilt in the asymmetric Laplace cases.

\section{Hedging proper losses under distribution shift
}\label{sec:hedging}

It is commonly held that proper loss functions encourage honest forecasting, in
the sense that if a forecaster believes the target distribution is $G$, then
they can minimize their expected loss by forecasting $G$. This statement can be
understood in the framework of subjective probability: if $G$ is the subjective
probability of the forecaster, then $G$ is the conditional distribution of the
target given the information available to the forecaster. Propriety guarantees
that no other forecast $F$ will incur lesser (conditional expected) loss,
$\ell(F,G) \geq \ell(G,G)$. This subjective probability formulation is
articulated explicitly in earlier seminal papers (as in
\citealp{brier1950verification, winkler1968good}), but not as often in modern
ones (as in \citealp{gneiting2007strictly, parry2012proper}). In practice,
forecasts are rarely subjective probabilities. This can be attributed to a
multitude of issues, such as misspecification of the data generating process. If
$G$ is not the forecaster's subjective probability, then no formal guarantee
appears to follow from the definition of propriety that would merit the
assertion that the forecaster should forecast their opinion $G$.

Consider a forecaster with firm opinion that the target distribution is $G$. The
forecaster may ask---what if I am wrong? If error is conceivable, the forecaster
might then ask---is it better to err in one form over another? If the answer to
the latter question is affirmative, the forecaster might be incentivized to
deviate from their honest opinion to protect against the possibility of
erring in a particularly unfavorable way. Such behavior has been termed
\emph{hedging}. A key motivation for using proper losses has historically
been the supposition that they provide no opportunity for hedging
\citep{brier1950verification, murphy1967note, winkler1968good}.

We are motivated to scrutinize this supposition in situations where forecasting
error is possible (or even unavoidable) and losses exhibit asymmetric penalties
on forecast errors. In particular, we study hedging in a setting of distribution
shift, in which the parameter in a scale or exponential family shifts in testing
relative to the training population. In this case, if the shift is log-symmetric
around the training parameter value, then the asymmetry in loss penalties can be
used to describe the direction toward which the forecaster ought to deviate from
their honest opinion.

\subsection{Problem setting}

We suppose that the forecaster has ideal knowledge of the training population,
say, by having access to infinitely many independent and identically distributed
observations, so they are able to conclude with certainty that the distribution
of the target in training is \smash{$G_{\sigma_{\text{train}}}$}. In testing,
however, suppose that there is distribution shift in the scale parameter, where
the law of target variable $Y$ in testing is defined as follows: we first draw a
random scale parameter \smash{$\sigma_{\text{test}}>0$}, then draw \smash{$Y
  \sim G_{\sigma_{\text{test}}}$}. To express our ignorance of the nature of the
distribution shift relative to the training population, we assume that
\smash{$\sigma_{\text{test}}$} is log-symmetric around
\smash{$\sigma_{\text{train}}$}, that is, \smash{$\log\sigma_{\text{test}}$}
has a symmetric distribution around \smash{$\log\sigma_{\text{train}}$}.

A forecaster not aware of distribution shift in their data might naively
forecast \smash{$G_{\sigma_{\text{train}}}$}. The question then arises whether
the forecaster can do better if they suspect distribution shift is present (as
would be common in many, if not most, real applications of forecasting).
If a loss function $\ell$ is proper, then the minimum expected loss over all
possible realizations of the test scale \smash{$\sigma_{\text{test}}$} is of
course obtained by \smash{$\E G_{\sigma_{\text{test}}}$}, the unconditional
law of $Y$. That is,
\[
\E \ell(F,G_{\sigma_{\text{test}}}) \geq \E \ell(\E G_{\sigma_{\text{test}}},
G_{\sigma_{\text{test}}}), \quad \text{for any $F$}.
\]
However, this requires the forecaster to have perfect knowledge of the
distribution of \smash{$\sigma_{\text{test}}$}, whereas our setting assumes no
such knowledge. The forecaster, equipped only with knowledge of log-symmetry of
the test scale distribution, might seek a better forecast than
\smash{$G_{\sigma_{\text{train}}}$} within the scale family
$\{G_\sigma:\sigma>0\}$, and the question we consider here is whether there is a
fixed forecast \smash{$G_{\sigma^*}$} with lesser expected loss. We show that
the answer is generally yes.

Before going on to state and prove general results about hedging for the various
loss functions, we present an informal discussion that may serve as an intuitive
guide relating hedging to asymmetric penalties. The setting of distribution
shift we analyze here differs from the previous setting of asymmetric penalties
studied in Section \ref{sec:main-results}, in that the roles of forecast and
target distributions are, in a certain sense, reversed. In Section
\ref{sec:main-results}, recall, the target $G = G_1$ was fixed and we considered
multiple forecasts $G_a$ and $G_{1/a}$ dispersed log-symmetrically around the
target. Instead, we now think of a given forecast $G$ as fixed, and we consider
multiple targets dispersed log-symmetrically around the forecast. If
\smash{$\sigma_{\text{test}}$} is distributed uniformly on $\{a,1/a\}$ with
$a>1$ (assuming without loss of generality that
\smash{$\sigma_{\text{train}}=1$}), then
\[
2\E d(G, G_{\sigma_{\text{test}}}) = d(G, G_a) + d(G, G_{1/a}).
\]
If the divergence is scale invariant, such as the one induced by logarithmic
loss, then
\[
d(G,G_a)=d(G_{1/a},G) \quad \text{and} \quad d(G,G_{1/a})=d(G_a,G).
\]
Meaning, penalties are reversed, and based on Theorem
\ref{thm:exponential-family} the realization $G_a$ of larger target scale will
typically lead to larger divergence. If instead the divergence is symmetric,
such as the one induced by CRP or quadratic loss, then
\[
d(G,G_a)=d(G_a,G) \; \; \; \; \text{and} \; \; \; \; d(G,G_{1/a})=d(G_{1/a},G).
\]
Lemma \ref{lem:scale-asymmetry} tells us that $G_a$ will lead to larger
divergence if $d$ is rescalable with increasing scale function, as in CRP loss,
or will lead to smaller divergence if $d$ is rescalable with decreasing scale
function, as in quadratic loss. In summary, for symmetric divergences the two
settings of asymmetric penalties and of distribution shift are interchangeable,
whereas for scale invariant divergences they are the reverse of each other.

It may appear at this point that a forecaster would hedge in a direction that
draws toward the least favorable realization of the test scale parameter in
order to mitigate it, and this is indeed the case for logarithmic, CRP, and
quadratic losses in certain cases. However, in general, the precise arguments
needed for these (and other) losses are more subtle and we give the details
across the next two subsections.

\subsection{Scale families}

The next theorem characterizes the direction of hedging for losses that induce
symmetric rescalable divergence, and the direction is shown to depend on the
particular scale family. (We make the assumption that the training scale
parameter is set to be $\sigma_{\text{train}}=1$ here, which we again emphasize
comes at no loss of generality.)

\medskip
\begin{theorem}
\label{thm:hedge-scale}
Let $\{G_\sigma:\sigma>0\}$ be a scale family, $\ell$ a loss function, and $d$
the induced divergence, where $d$ is symmetric and rescalable with a scaling
function $h$. Let $\sigma_{\text{test}}>0$ be a random variable, which is not
almost surely constant, such that $\sigma_{\text{test}}$ and
$1/\sigma_{\text{test}}$ have the same distribution and $\E \ell(G_\sigma,
G_{\sigma_{\text{test}}})$ is finite for all $\sigma>0$. Fix $G=G_1$ and define
\[
f(\sigma)=d(G_{\sigma}, G),
\]
Assume that $f$ and $h$ are differentiable and conditions for exchanging
differentiation and taking expectation with respect to $\sigma_{\text{test}}$
apply (e.g., a sufficient condition is that $\sigma_{\text{test}}$ has
compactly supported density and $f'$ is continuous). Then the following holds.
\begin{enumerate}
\item If $(h-1)/f$ is increasing on $[1,\infty)$, then there exists $\sigma^*<1$
  such that $\E \ell(G_{\sigma^*}, G_{\sigma_{\text{test}}}) < \E \ell(G,
  G_{\sigma_{\text{test}}})$.

\item If $(h-1)/f$ is decreasing on $[1,\infty)$, then there exists $\sigma^*>1$
  such that $\E \ell(G_{\sigma^*}, G_{\sigma_{\text{test}}}) < \E \ell(G,
  G_{\sigma_{\text{test}}})$.
\end{enumerate}
\end{theorem}

\begin{proof}
Define
\[
g_a(\sigma) = d(G_\sigma, G_a) + d(G_{\sigma}, G_{1/a}).
\]
We will first show that $g_a'(1)$ has the same sign as the derivative of
$(h-1)/f$ at $a$. Indeed, using rescalability (twice), we can rewrite
$g_a(\sigma)$ as
\[
g_a(\sigma) = h(a)d(G_{\sigma/a}, G) + \frac{1}{h(a)}d(G_{a\sigma}, G) = h(a)
f(\sigma/a) + \frac{1}{h(a)} f(a\sigma).
\]
Since $f$ is differentiable, so is $g_a$, and the derivative of $g_a$ at $1$ is
\[
g_a'(1) = \frac{h(a)}{a}f'(1/a) + \frac{a}{h(a)}f'(a).
\]
Now we relate $f'(1/a)$ to $f'(a)$ in order to plug into the equation for
$g_a'(1)$ above. By symmetry and rescalability, note that
$f(\sigma)=h(\sigma)f(1/\sigma)$, and differentiating this we get
\[
f'(\sigma) = h'(\sigma)f(1/\sigma) - \frac{h(\sigma)}{\sigma^2}f'(1/\sigma) =
\frac{h'(\sigma)}{h(\sigma)}f(\sigma) - \frac{h(\sigma)}{\sigma^2}
f'(1/\sigma).
\]
Rearranging to isolate $f'(1/\sigma)$, and plugging into the previous equation
for $g_a'(1)$, we arrive at
\[
g_a'(1) = \frac{a}{h(a)} (h'(a)f(a) - (h(a)-1)f'(a)) =
\frac{af(a)^2}{h(a)} \bigg(\frac{h-1}{f}\bigg)'(a).
\]
The condition that $(h-1)/f$ is increasing on $[1, \infty)$ is thus sufficient
to ensure a positive derivative $g_a'(1)>0$ for all $a>1$. Recalling
$\sigma_{\text{test}}$ is a random variable such that $\sigma_{\text{test}}$ and
$1/\sigma_{\text{test}}$ have a common distribution, denoted $H$, define
\[
g(\sigma) = 2\E d(G_{\sigma}, G_{\sigma_{\text{test}}}) = \E[d(G_\sigma,
G_{\sigma_{\text{test}}}) + d(G_{\sigma}, G_{1/\sigma_{\text{test}}})] =
\int g_a(\sigma) \, dH(a).
\]
If differentiation under the integral sign is permitted, then
\[
g'(\sigma) = \int g_a'(\sigma) \, dH(a).
\]
The condition that $(h-1)/f$ is increasing on $[1,\infty)$ is thus also
sufficient to ensure $g'(1)>0$, entailing the existence of $\sigma^*<1$ for
which
\[
\E \ell(G_{\sigma^*}, G_{\sigma_{\text{test}}}) - \E \ell(G,
G_{\sigma_{\text{test}}}) = \frac{1}{2}(g(\sigma^*)-g(1)) < 0.
\]
The case where $(h-1)/f$ is decreasing is proven similarly.
\end{proof}

\begin{example}[Exponential distribution]
When $G_\sigma$ is the exponential distribution with scale $\sigma$ (see Section
\ref{sec:main-exponential}), we can infer the following using Theorem
\ref{thm:hedge-scale}. Under CRP loss, hedging is carried out by inflating the
scale, making it flatter and less informative, while under quadratic loss it is
carried out by deflating the scale, making it sharper and overconfident. To see
this, we note a simple calculation yields for CRP loss that
\[
d(G_\sigma, G) = (1+\sigma)/2 - 2\sigma/(1 + \sigma).
\]
Recalling $h(\sigma)=\sigma$, it may be shown that $d(G_\sigma,
G)/(h(\sigma)-1)$ is increasing. Hence by Theorem \ref{thm:hedge-scale}, there
is a forecast $G_{\sigma^*}$ flatter than $G$ (i.e., $\sigma^*>1$) which attains
lower expected CRP loss. For quadratic loss, a simple calculation yields
\[
d(G_\sigma, G)=(\sigma-1)^2/(2\sigma(\sigma+1)).
\]
Recalling $h(\sigma)=1/\sigma$, it may be shown that
$(h(\sigma)-1)/d(G_\sigma,G)$ is increasing. Thus by Theorem
\ref{thm:hedge-scale}, there is a forecast $G_{\sigma^*}$ sharper than $G$
(i.e., $\sigma^*<1$) which attains lower expected quadratic loss.
\end{example}

\subsection{Exponential families}

For logarithmic loss, the divergence it induces is not symmetric, but we can
characterize the direction of hedging with a specialized argument for
exponential families.

\medskip
\begin{theorem}
\label{thm:hedge-exponential}
Let  $\{p_\eta:\eta>0\}$ be a minimal exponential family of densities, where
$p_\eta(x)=h(x) e^{\eta T(x) -A(\eta)}$. Denote by $\ell$ the logarithmic loss,
and let $\eta_{\text{test}}>0$ be a random variable where we assume $\E T(Y)$
exists and is finite, where $Y$ is a random variable whose conditional
distribution given $\eta_{\text{test}}$ is $p_{\eta_{\text{test}}}$. Then
$\eta^*=(A')^{-1}(\E A'(\eta_{\text{test}}))$ is well-defined and minimizes then
the expected loss $\E \ell(p_\eta, p_{\eta_{\text{test}}})$ over all $\eta>0$.
\end{theorem}

\begin{proof}
We first show that $\eta^*$ is well-defined. Recall, for a minimal exponential
family, the log-partition function $A$ is continuously differentiable and
strictly convex \citep{wainwright2008graphical}, and $A'$ acts as a bijection
between its domain and image \citep{rockafellar1970convex}, which must then be
an open interval. Furthermore, by a standard identity for exponential families,
\[
\E T(Y) = \E[\E T(Y)|\eta_{\text{test}}] = \E A'(\eta_{\text{test}}).
\]
By assumption, this value is finite, and thus the expectation $\E
A'(\eta_{\text{test}})$ must lie within the image of $A'$, which implies the
existence and uniqueness of $\eta^*$ for which $A'(\eta^*)=\E
A'(\eta_{\text{test}})$. It remains to show that $\eta^*$ minimizes the expected
logarithmic loss by completing the Bregman divergence. We can compute the
expected loss for any $\eta>0$ by
\[
\E\ell(p_\eta, p_{\eta_{\text{test}}}) = -\E[\E\log
p_\eta(Y)|\eta_{\text{test}}] = A(\eta) - \eta \E T(Y) + c
= A(\eta) - \eta A'(\eta^*) + c,
\]
where $c$ is constant. The difference in expected loss between $\eta$ and
$\eta^*$ is therefore
\[
\E \ell(p_\eta, p_{\eta_{\text{test}}}) - \E \ell(p_{\eta^*},
p_{\eta_{\text{test}}}) = A(\eta)-A(\eta^*) - (\eta-\eta^*)A'(\eta^*).
\]
This is the Bregman divergence of the strictly convex function $A$, tangent to
$A$ at $\eta^*$, which is nonnegative and vanishes if and only if $\eta=\eta^*$.
\end{proof}

No assumption was required so far on the relation between $\eta_{\text{train}}$
and $\eta_{\text{test}}$, the natural parameters in training and testing. If we
assume $\eta_{\text{test}}$ is logarithmically symmetric around
$\eta_{\text{train}}$, then
\[
\eta_{\text{train}} = e^{\E \log(\eta_{\text{test}})}.
\]
Compare this with the optimal forecast from the theorem:
\[
\eta^* = (A')^{-1}(\E A'(\eta_{\text{test}})).
\]
We see that whether $\eta^* < \eta_{\text{train}}$ or $\eta^* >
\eta_{\text{train}}$ holds will depend on the log-partition function $A$. We
conclude this subsection with an example where this occurs.

\medskip
\begin{example}[Generalized gamma scale family]
When $p_\eta$ is the generalized gamma density (see Section
\ref{sec:main-exponential}), with scale $\sigma$ inversely proportional to
$\eta$, we can infer the following from Theorem \ref{thm:hedge-exponential}.
Under logarithmic loss, hedging is carried out by inflating the scale relative
to the training population. To see this, recall that for the generalized gamma
family $\eta=1/\sigma^\gamma$ with $\gamma>0$, and the log-partition
function has derivative $A'(\eta)=-(k/\gamma)/\eta$. The natural parameter
$\eta^*$ for which the forecast $p_{\eta^*}$ attains global minimum expected
logarithmic loss is
\[
\frac{1}{\eta^*} = \E \left[\frac{1}{\eta_{\text{test}}}\right].
\]
Now compare the optimal scale $\sigma^*=(\eta^*)^{-1/\gamma}$ with the training
scale $\sigma_{\text{train}}=(\eta_{\text{train}})^{-1/\gamma}$: using Jensen's
inequality,
\[
(\sigma^*)^{-\gamma} = \eta^* = (\E \eta_{\text{test}}^{-1})^{-1}
< e^{\E\log(\eta_{\text{test}})} = \eta_{\text{train}} =
\sigma_{\text{train}}^{-\gamma},
\]
which shows that $\sigma^*>\sigma_{\text{train}}$. From special cases of the
generalized gamma scale family, we may derive results for the exponential,
Laplace, normal, gamma, and Weibull scale families.
\end{example}

\section{Discussion}
\label{sec:discussion}

In this work, we studied asymmetries in the penalization of a broad set of
proper loss functions, including logarithmic, continuous ranked probability,
threshold-weighted CRP, quadratic, spherical, energy, and Dawid--Sebastiani
losses. To recap some highlights, by establishing general results in exponential
families for logarithmic loss, we showed this loss typically penalizes
overestimating scale parameters less severely than underestimating them by the
same amount on a logarithmic scale. Moreover, by introducing the notion of
symmetric rescalable divergences, we showed that in scale families CRP loss
favors sharp forecasts (underestimating the scale), whereas quadratic loss
favors flat forecasts (overestimating the scale). These results are clearly
visible in practice: through experiments, we confirmed the effects anticipated
by the theory on data from Covid-19 mortality, temperature, and retail
forecasts. Finally, under a setting with distribution shift, we showed that
hedging of certain proper loss functions is possible, which can be understood as
an implication of their inherent asymmetry.

We close with some additional related comments and discussion.

\subsection{Confounding effects of aggregation across different scales}

In practice, a loss is often averaged over non-identically distributed target
observations, for example, to evaluate average performance of a forecaster
across different dates, geographic locations, or generally over different tasks.
The problem of confounding effects stemming from the differences between the
tasks is well-recognized, with skill losses a popular yet imperfect remedy
\citep{gneiting2007strictly}. Here we note that the confounding effects of scale
can depend on the loss function being used, and the asymmetries therein.

We demonstrate via two examples that losses which induce symmetric rescalable
divergences with increasing scaling functions (e.g., CRP and energy) place more
weight on observations with large scale versus those with small scale, and
losses with decreasing scaling function (e.g., quadratic) behave in the opposite
way. On the other hand, losses with constant scaling function (e.g., logarithmic
and DS) are indifferent to the scale of the target.

\medskip
\begin{example}[Different specializations]
Consider two tasks with different scale, and two forecasters, Glenn and
Bob. Glenn is relatively good in forecasting one task, and Bob equally better in
the other. We show that which forecaster is awarded least expected loss depends
on the loss function being used. Concretely, in one task with the target
distribution being $G$, suppose that Glenn forecasts $F$ and Bob $H$, and Glenn
achieves lower expected loss,
\[
\ell(F,G) < \ell(H,G).
\]
However, in another task with target distribution $G_\sigma$, suppose that in a
reversal Bob now forecasts $F_\sigma$ and Glenn $H_\sigma$. Here, the
distributions $F_\sigma,G_\sigma,H_\sigma$ should be read as members of
respective scale families, where we assume $\sigma>1$ to signify increased
scale. If the divergence induced by $\ell$ is symmetric and rescalable with
constant scaling function (e.g., logarithmic and DS), it is indifferent to the
scale of the second task and assigns equal expected loss to Glenn and Bob in
total over both tasks,
\[
\ell(F,G)+\ell(H_\sigma,G_\sigma) = \ell(H,G)+\ell(F_\sigma,G_\sigma).
\]
However, when the divergence has increasing scaling function (e.g., CRP and energy),
more emphasis is placed on the second upscaled task, and consequently Bob wins:
\[
\ell(F,G)+\ell(H_\sigma,G_\sigma) > \ell(H,G)+\ell(F_\sigma,G_\sigma).
\]
Lastly, if the divergence has decreasing scaling function (e.g., quadratic, and also
spherical when $F,G,H$ belong to the same scale family), lesser weight is put on
the second task due to its increased scale, and Glenn wins:
\[
\ell(F,G)+\ell(H_\sigma,G_\sigma) < \ell(H,G)+\ell(F_\sigma,G_\sigma).
\]
We have reached three different conclusions by using different loss functions,
with all else being equal.
\end{example}

\medskip
\begin{example}[Missing forecasts]
In this example, Glenn and Bob make the same forecasts, but Bob is missing
forecasts for some target observations. (Missingness is common in some domains,
e.g., in epidemiological forecasting; \citet{cramer2022eval} report that only
28 out of 71 forecasters of Covid-19 mortality submitted full forecasts for at
least 60\% of participating weeks in their analysis.) Concretely, in the first
task, suppose that both Glenn and Bob forecast $F$, when the target distribution
is $G$. In the second task, suppose Glenn forecasts $F_\sigma$, and the
target distribution is $G_\sigma$ with $\sigma>1$, however, Bob makes no
forecast. Which forecaster is awarded the least expected loss---now averaged
over the observed forecasts---depends on the loss function being used. If the
divergence induced by $\ell$ is symmetric and rescalable with constant scaling
function (e.g., logarithmic and DS), then expected loss at the second task
equals that at the first task,
\[
\frac{\ell(F,G)+\ell(F_\sigma,G_\sigma)}{2} = \ell(F,G),
\]
and neither Glenn nor Bob wins. If, however, the divergence has increasing scaling
function (e.g., CRP and energy), then the loss at the second task is greater,
leading to Bob winning:
\[
\frac{\ell(F,G)+\ell(F_\sigma,G_\sigma)}{2} > \ell(F,G).
\]
Conversely, if the divergence has decreasing scaling function (e.g., quadratic, and
also spherical when $F,G$ belong to the same scale family), then the loss at the
second task is lesser, leading to Glenn winning:
\[
\frac{\ell(F,G)+\ell(F_\sigma,G_\sigma)}{2} < \ell(F,G).
\]
Once again, we see that three different conclusions have been reached by using
different loss functions.

\end{example}

\subsection{A closer look at logarithmic loss in exponential families}

In Section \ref{sec:main-exponential}, we considered densities $p_\eta(x)=h(x)
e^{\eta T(x)-A(\eta)}$ in the exponential family $\{p_\eta: \eta>0\}$ and
proved, for $\theta > 1$ and $\ell$ being the logarithmic loss, that
$\ell(p_{\theta\eta}, p_\eta) - \ell(p_{\eta/\theta}, p_\eta)$ is positive,
negative, or zero when $\eta^3A''(\eta)$ is respectively increasing, decreasing,
or constant. Instead of setting a parametrization a priori (i.e., comparing
$\theta\eta$ to $\eta/\theta$), one may instead ask about the sign of
$\ell(p_{\eta_1}, p_\eta) - \ell(p_{\eta_2}, p_\eta)$ for varying $\eta_1$, with
$\eta_2,\eta$ fixed. Figure \ref{fig:gaussian} visualizes this quantity for the
normal scale family, and Theorem \ref{thm:exponential-root} provides a precise
characterization for a wide class of distributions, which includes the normal
scale family, log-normal log-scale family, exponential family, Weibull scale
family, Laplace scale family, gamma scale family, and others.

\begin{figure}[htb]
\centering
\includegraphics[width=\textwidth]{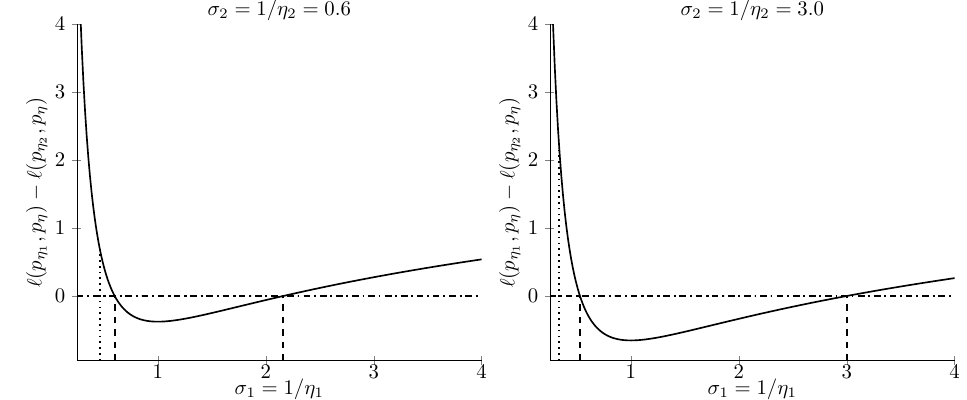}
\caption{\small The function $\ell(p_{\eta_1}, p_\eta) - \ell(p_{\eta_2},
  p_\eta)$ for normal densities with distinct fixed values of $\sigma_2$
  (solid), the roots given by Theorem \ref{thm:exponential-root} in terms of the
  Lambert function (dashed) and the multiplicative inverse of the largest of the
  two roots (dotted).}
\label{fig:gaussian}
\end{figure}

\begin{theorem}
\label{thm:exponential-root}
For an exponential family $\{p_\eta: \eta>0\}$ with log-partition
function of the form $A(\eta) = c_1 \log\eta + c_2$, for constants
$c_1,c_2\in\R$, the roots of $\ell(p_{\eta_1}, p_\eta) - \ell(p_{\eta_2},
p_\eta)$ occur at
\[
-\eta W\left(-\frac{\eta_2}{\eta} \exp\left(-\frac{\eta_2}{\eta}\right)\right),
\]
where $W\colon [-\frac{1}{e}, \infty) \to \R^2$ is the Lambert function, which
satisfies the implicit equation $x = W(x) \exp(W(x))$.
\end{theorem}

\begin{proof}
Representing the difference of logarithmic losses by the equivalent difference
of KL divergences yields
\[
\ell(p_{\eta_1}, p_\eta)-\ell(p_{\eta_2}, p_\eta)
  = d(p_{\eta_1}, p)-d(p_{\eta_2}, p_\eta)
  = \E \log\frac{p_{\eta_2}(Y)}{p_{\eta_1}(Y)}
  = A(\eta_1) - A(\eta_2) - (\eta_1-\eta_2)\E T(Y),
\]
where $Y$ is a random variable with density $p_\eta$. From the exponential
family identity $\E_\eta T(Y)=A'(\eta)$, we have
\[
\ell(p_{\eta_1}, p_\eta)-\ell(p_{\eta_2}, p_\eta)
  = A(\eta_1) - A(\eta_2) - (\eta_1-\eta_2) A'(\eta)
  = c_1\log\eta_1 - c_1\log\eta_2 - \frac{c_1}{\eta} (\eta_1-\eta_2).
\]
Setting the above display to zero gives $\log\frac{\eta_1}{\eta_2} =
\frac{\eta_1-\eta_2}{\eta}$. After rearranging and taking exponents, we have
\[
\exp(\eta_1) \exp(-\eta \log\eta_1) = \exp(\eta_2) \exp(-\eta \log\eta_2).
\]
Exponentiating by $-\frac{1}{\eta}$ and then multiplying by $-\frac{1}{\eta}$,
gives
\[
-\frac{\eta_1}{\eta} \exp\left(-\frac{\eta_1}{\eta}\right)
  = -\frac{\eta_2}{\eta} \exp\left(-\frac{\eta_2}{\eta}\right),
\]
from which the Lambert function is immediately recognized, yielding
\[
-\frac{\eta_1}{\eta} = W\left(-\frac{\eta_2}{\eta}
  \exp\left(-\frac{\eta_2}{\eta}\right) \right).
\]
The result follows by multiplying both sides by $-\eta$.
\end{proof}

\subsection{Threshold-weighted continuous ranked probability loss}
\label{sec:tw-crp}

In the introduction and throughout the paper, we alluded to results available
for threshold-weighted CRP loss. Here we present the details, for weights of the
form $w(y) = y^\alpha$ with $\alpha\in\R$, subject to $w$ remaining nonnegative
(power-weighted CRP loss). This loss has asymmetries in location and scale
families governed by the exponent $\alpha$. Indeed, as a direct consequence of
Lemma \ref{lem:rescalability}, for $\{G_\sigma:\sigma>0\}$ a scale family,
$\sigma>0$, and $\ell$ a power-weighted CRP loss, we have:
\begin{itemize}
\item $\ell(G_\sigma, G) = \ell(G_{1/\sigma}, G)$ if $\alpha = -1$;
\item $\ell(G_\sigma, G) > \ell(G_{1/\sigma}, G)$ if $\alpha > -1$;
\item $\ell(G_\sigma, G) < \ell(G_{1/\sigma}, G)$ if $\alpha < -1$.
\end{itemize}

Moreover, for $\{G_\mu:\mu\in\R\}$ a location family, by a similar argument,
we have:
\begin{itemize}
\item $\ell(G_\mu, G) = \ell(G_{-\mu},G)$ if $\alpha = 0$ or $\mu = 0$;
\item $\ell(G_\mu, G) > \ell(G_{-\mu},G)$ if $\sign(\alpha) = \sign(\mu) \neq
  0$;
\item $\ell(G_\mu, G) < \ell(G_{-\mu},G)$ otherwise.
\end{itemize}

We conclude that power-weighted CRP loss has an inherent trade-off: it can
penalize symmetrically in location families at the expense of asymmetry in scale
families ($\alpha = 0$), or it can penalize symmetrically in scale families at
the expense of asymmetry in location families ($\alpha = -1$). There is no power
weight function that guarantees symmetric penalties in both families
simultaneously. A natural question is now whether there exists a weight
function, not necessarily a power function, that wholly symmetrizes CRP loss in
this context. Our next results answers this in the negative, apart from the
trivial zero function.

\medskip
\begin{proposition}
Let $\{G_\sigma:\sigma>0\}$ and $\{G_\mu:\mu\in\R\}$ be scale and location
families, respectively. For the threshold-weighted CRP loss $\ell$,
\[
\ell(G_\sigma,G_1) = \ell(G_{1/\sigma},G_1) \; \text{and} \; \ell(G_\mu,G_0) =
\ell(G_{-\mu},G_0), \; \text{for all $\sigma>0,\mu\in\R$} \iff w(y) = 0, \;
\text{for all $y \in \cY$}.
\]
In other words, there does not exist a nonzero weight function such that
symmetry is achieved for a location and scale family simultaneously.
\end{proposition}

\begin{proof}
The ``if'' direction is obvious. For the ``only if'' direction, observe that
$\ell(G_\sigma, G) = \ell(G_{1/\sigma}, G)$ implies
\[
\int_\cY \sigma w(y\sigma) (G(y) - G(y\sigma))^2 dy = \int_\cY w(y)
(G(y) - G(y\sigma))^2 dy \implies \sigma w(y\sigma) = w(y),
\]
whereas $\ell(G_\mu, G) = \ell(G_{-\mu}, G)$ implies
\[
\int_\cY w(y+\mu) (G(y) - G(y+\mu))^2 dy = \int_\cY w(y) (G(y) - G(y+\mu))^2 dy
\implies w(y+\mu) = w(y).
\]
The only weight function that satisfies both conditions for all $\sigma > 0$ and
$\mu \in \R$ is the zero function: $w(y) = 0$ for all $y \in \cY$.
\end{proof}

Therefore, the trade-off between ensuring symmetric penalties either in location
or in scale families is not a peculiarity of power-weighted CRP loss, but
extends to the more general threshold-weighted version. In this particular
sense, CRP loss is unsymmetrizable. Exploring whether quantile-weighted CRP loss
behaves similarly remains an avenue for future work.

\subsection*{Acknowledgements}

We thank members of the Delphi research group, as well as Tilmann Gneiting, Evan
Ray, and Nicholas Reich for useful discussions. We also thank the participants
of the Berkeley Statistics probabilistic forecasting reading group. This work
was supported by Centers for Disease Control and Prevention (CDC) grant number
75D30123C15907.

{\RaggedRight
\bibliographystyle{plainnat}
\bibliography{refs,general}}

\appendix
\section{More details on Covid-19 mortality experiments}

\subsection{Converting Hub forecasts}
\label{app:covid-converting}

Forecasts in the Hub appear in terms of their quantiles
\smash{$[F]=\{F^{-1}(\tau):\tau\in T\}$}, where $F$ is the forecasted cumulative
distribution function and $T$ is a discrete set of probability levels,
containing evenly-spaced values from 0.05 to 0.95 in increments of 0.05, as well
as 0.01, 0.025, 0.0975, and 0.99. Thus a forecast is in fact not a probability
distribution but an equivalence class $[F]$ over probability distributions
comprising all distributions that identify on every quantile in $T$. In order to
compute the loss for a given forecast and observed outcome, we choose a
particular representative from the equivalence class $[F]$, described as
follows. First, we set the representative to have a density which is piecewise
linear between $\min [F]$ and $\max [F]$ with knots at elements of
$[F]$. Second, we set the representative to have lower and upper tails (below
$\min [F]$ and above $\max [F]$, respectively) of exponential distributions with
quantiles matching $[F]$ on the bottom and top two quantiles in $T$,
respectively. If $F$ is the representative of $[F]$ described above, then we
define the loss of the forecast to be $\ell([F], y)=\ell(F, y)$. Likewise, the
forecast variance was defined as the variance of a random variable with
distribution $F$. Finally, we discarded forecasts from the Hub with atoms
(quantiles which were equal at adjacent probability levels), or forecasts with
quantile crossings (quantiles which were out of order at adjacent levels).

\subsection{Standardizing target values}
\label{app:covid-standardizing}

For each location, we estimated the smoothed mean function of the target
distribution across time using trend filtering of cubic order
\citep{kim2009trend, tibshirani2014adaptive} on the observed death counts. Trend
filtering is a general-purpose nonparametric smoother that acts similarly to a
locally-adaptive regression spline; it is formulated as the solution to a
penalized least squares problem, and we used a cross-validation scheme with the
one-standard-error rule to choose the regularization parameter. We then
estimated the smoothed variance of the target distribution across dates by
applying trend filtering of constant order to the squared residuals from the
first step (observed outcomes minus smoothed means), again using
cross-validation to select the regularization parameter.

\end{document}